\documentclass[11pt,thmsa]{article}

\usepackage{amssymb,latexsym}

\usepackage{amsmath,amscd}

\usepackage[all]{xy}

 \usepackage[latin1]{inputenc}

 \usepackage{mathrsfs}
 \usepackage{amsmath}
 \usepackage{amsthm}
 \usepackage{amsfonts}
 \usepackage{amssymb}
 \usepackage[pdftex]{graphicx}
 \usepackage{wrapfig}
 \usepackage{layout}
 \usepackage{verbatim}
 \usepackage{alltt}
 \usepackage{xypic}
 \usepackage{tikz}
\usetikzlibrary{decorations.markings}
 \input xy
 \xyoption{all}

 \DeclareMathAlphabet{\mathpzc}{OT1}{pzc}{m}{it}

 \newtheorem{theorem}{Theorem}[section]
 \newtheorem{lemma}[theorem]{Lemma}
 \newtheorem{proposition}[theorem]{Proposition}

 \newtheorem{definition}[theorem]{Definition}
  \theoremstyle{definition}
 
 \newtheorem{remark}[theorem]{Remark}

\newtheorem*{conventions}{Conventions}

\newtheorem*{theorem*}{Theorem \ref{main}}
\renewenvironment{proof}{\noindent{\it
Proof.}}{\bgroup\hspace{\stretch{1}}$\square$\egroup\medskip\par}

\newcommand{\Rep}{\textrm{Rep}}

\newcommand{\id}{\mathrm{id}}
\newcommand{\End}{\mathrm{End}}
\newcommand{\Hom}{\textrm{Hom}}
\newcommand{\RHom}{\underline{\mathrm{Hom}}}

\newcommand{\front}{P}
\newcommand{\back}{Q}
\newcommand{\A}{\mathsf{A}}
\newcommand{\D}{{\bf \Delta}}
\newcommand{\Hol}{\mathsf{Hol}}
\newcommand{\Flat}{\mathsf{Flat}}
\begin{document}

% Macro for drawing a tetrahedron
\def\tetra{\draw(A) -- (B) -- (C) -- (D) -- cycle; \draw[shade](A)--(B)--(D); \draw[shade](B)--(C)--(D);\draw(B)--(D); \draw[dotted](A)--(C);\draw[dotted](A)--(C); \draw(A)--(D);}
\def\tetraA{\draw(A) -- (B) -- (C) -- (D) -- cycle; \draw[shade=gray](B)--(C)--(D);\draw[thin](B)--(D); \draw[dotted](A)--(C);\draw[dotted](A)--(C); \draw[thick](A)--(D);}
\def\tetraB{\draw(A) -- (B) -- (C) -- (D) -- cycle; \shade(A)--(C)--(D);\draw(B)--(D); \draw(A)--(D); \draw[dotted](A)--(C);\draw(C)--(D);}
\def\tetraC{\draw(A) -- (B) -- (C) -- (D) -- cycle; \draw[shade=gray](A)--(B)--(D);\draw(B)--(D); \draw[dotted](A)--(C); \draw[dotted](A)--(C); \draw(A)--(D);\draw[thick](B)--(C);}
\def\tetraD{\draw(A) -- (B) -- (C) -- (D) -- cycle; \draw[shade=gray](A)--(B)--(C);\draw(B)--(D); \draw[dotted](A)--(C);\draw[dotted](A)--(C);\draw(A)--(D);  }

% Macro for drawing a triangle
\def\triangle{\draw(A) -- (B) -- (C)--cycle; \draw[shade](A)--(B)--(C);\draw(A)--(C);}
\def\triangleA{\draw(A) -- (B) -- (C)  -- cycle;\draw[ultra thick](A)--(C);\draw[ultra thick](C)--(B);}
 \def\triangleB{\draw(A) -- (B) -- (C)  -- cycle;\draw[ultra thick](A)--(B);}

\def\triangleabd{\draw(a)--(b)--(d) --cycle;}
\def\trianglebde{\draw(b)--(d)--(e) --cycle;}
\def\trianglecbe{\draw(a)--(d)--(e) --cycle;}
\def\trianglecef{\draw(a)--(b)--(c) --cycle;}

\def\triangleabc{\draw(c)--(d)--(f) --cycle;}
\def\triangleacd{\draw(a)--(c)--(d) --cycle;}
\def\triangleabe{\draw(a)--(b)--(e) --cycle;}
\def\triangleade{\draw(a)--(d)--(e) --cycle;}

\def\trianglebce{\draw(b)--(c)--(e) --cycle;}
\def\trianglebcf{\draw(b)--(c)--(f) --cycle;}
\def\trianglebef{\draw(b)--(e)--(f) --cycle;}
\def\trianglecef{\draw(c)--(e)--(f) --cycle;}

\def\trianglebce{\draw(b)--(c)--(e) --cycle;}
\def\trianglebcf{\draw(b)--(c)--(f) --cycle;}
\def\trianglebef{\draw(b)--(e)--(f) --cycle;}
\def\trianglecef{\draw(c)--(e)--(f) --cycle;}

\def\trianglecdf{\draw(c)--(d)--(f) --cycle;}
\def\triangleacf{\draw(a)--(c)--(f) --cycle;}
\def\triangleadf{\draw(a)--(d)--(f) --cycle;}

\def\triangleaef{\draw(a)--(e)--(f) --cycle;}
\def\triangleabf{\draw(a)--(b)--(f) --cycle;}
\def\trianglecef{\draw(c)--(e)--(f) --cycle;}

\def\ssquare{\draw(a)--(c)--(f)--(d)--cycle;}
\def\splitsquare{\draw(a)--(c)--(f)--(d)--cycle; \draw(b)--(e);}

\def\cubefromabove{\draw(a)--(b)--(d)--(c)--cycle; \draw(A)--(B)--(D)--(C)--cycle;\draw(a)--(A);\draw(b)--(B);\draw(c)--(C); \draw(d)--(D);}

\def\completesquare{\draw(a)--(c)--(f)--(d)--cycle; \draw(b)--(e);\draw(b)--(d); \draw(c)--(e);}
\def\dsquare{\draw(a)--(c)--(f)--(d)--cycle;\draw(c)--(d);a}
%Macro for drawing a line
\def\vsplitsquare{\draw(a)--(b)--(f)--(e)--cycle; \draw(c)--(d);}
\def\line{\draw[thick](A) -- (B);\fill[ultra thick](A);}

% Macro for drawing brackets

\def\bleft{\draw(X) -- (Y) -- (Z)--(W);}
\def\bright{\draw(x) -- (y) -- (z)--(w);}
%%%%%%%%%%%%%%%%%%%%%%%%%%%%%%%%%%%%%%
%%%%%%%%%%%%%%%%%%%%%%%%%%%%%%%%%%%5

\def\Triangleabd{\draw(a)--(b)--(d) --cycle;\draw(d)--(f);}
\def\Triangleacd{\draw(a)--(c)--(d) --cycle;\draw(d)--(f);}
\def\Triangleabc{\draw(a)--(b)--(c) --cycle;\draw(c)--(d);\draw(d)--(f);}
\def\Trianglebcd{\draw(b)--(c)--(d) --cycle;\draw(d)--(f);\draw(a)--(b);}

\def\Trianglecdf{\draw(c)--(d)--(f) --cycle;\draw(a)--(c);}
\def\Trianglecde{\draw(c)--(d)--(e) --cycle;\draw(a)--(c);\draw(e)--(f);}
\def\Triangledef{\draw(d)--(e)--(f) --cycle;\draw(a)--(e);}
\def\Trianglecef{\draw(c)--(e)--(f) --cycle;\draw(a)--(c);}

\def\Triangleabe{\draw(a)--(b)--(e) --cycle;\draw(f)--(e);}
\def\Trianglebef{\draw(b)--(e)--(f) --cycle;\draw(b)--(a);}
\def\Triangleaef{\draw(a)--(e)--(f) --cycle;}
\def\Triangleabf{\draw(a)--(b)--(f) --cycle;}

\def\Triangleadf{\draw(a)--(d)--(f) --cycle;}
\def\Trianglebde{\draw(b)--(d)--(e) --cycle;\draw(b)--(e);}
\def\Triangleacf{\draw(a)--(c)--(f) --cycle;}
\def\Trianglebce{\draw(b)--(c)--(e) --cycle;\draw(e)--(f);}

\vspace{15cm}
 \title{Higher holonomies: comparing two constructions}

\author{Camilo Arias Abad\footnote{Max-Planck-Institut f\"ur Mathematik, Vivatsgasse 7, 53111 Bonn, Germany.  Partially supported
by SNF Grant 200020-131813/1 and a Humboldt  Fellowship. E-mail: camiloariasabad@gmail.com.} \hspace{0cm} and
Florian Sch\"atz\footnote{Centre for Quantum Geometry of Moduli Spaces, Ny Munkegade 118,
DK-8000 Aarhus C, Denmark. Supported by the Danish National Research Foundation grant DNRF95 (Centre for Quantum Geometry of Moduli Spaces - QGM). E-mail: florian.schaetz@gmail.com.}}

 \maketitle
 \begin{abstract} 
We compare two different constructions of higher dimensional parallel transport.
On the one hand, there is the two dimensional parallel transport associated to 2-connections on 2-bundles studied by Baez-Schreiber \cite{BaezS}, Faria Martins-Picken \cite{Picken} and Schreiber-Waldorf \cite{SchreiberWaldorf}. On the other hand, there are the higher holonomies associated to flat superconnections as studied by Igusa \cite{I}, Block-Smith \cite{BS} and Arias Abad-Sch\"atz \cite{AS}. 
We first explain how by truncating the latter construction one obtains examples of the former.
Then we prove that the 2-dimensional holonomies provided by the two approaches coincide.
\end{abstract}

\tableofcontents

\section{Introduction}\label{s:intro}

The purpose of this note is to compare two extensions of parallel transport to higher dimensional objects.
On the one hand, there is the parallel transport for flat $2$-connections with values in crossed modules studied by
 Baez-Schreiber \cite{BaezS}, Faria Martins-Picken \cite{Picken} and Schreiber-Waldorf \cite{SchreiberWaldorf}.
The fundamental result in this approach is the construction of $2$-dimensional holonomies. This construction yields a map
\[\mathsf{Hol}:\mathsf{Flat}(M,\mathsf{g})\to \mathrm{Rep}(\pi_{\le 2}(M),\mathsf{G}),\]
that assigns to any flat $2$-connection with values in the differential crossed module $\mathsf{g}$ a representation of the fundamental $2$-groupoid of $M$. Here $\mathsf{G}$ is a Lie crossed module whose infinitesimal counterpart is $\mathsf{g}$.

On the other hand, there is the parallel transport of flat superconnections introduced recently by Igusa \cite{I} and subsequently studied by Block-Smith \cite{BS} and 
Arias Abad-Sch\"atz \cite{AS}. The fundamental result in this direction is that there is a weak equivalence of $dg$-categories between the category $\mathsf{Rep}_\infty(TM)$ of flat superconnections on $M$ and the category of $\infty$-representations of
the $\infty$-groupoid $\mathsf{Rep}_\infty(\pi_{\infty}(M))$ of $M$. In particular, there is an integration $\A_\infty$-functor
\begin{align*}
\int: \mathsf{Rep}_\infty(TM) \to \mathsf{Rep}_\infty(\pi_{\infty}(M)),
\end{align*}
which associates to any flat superconnection on $M$ an $\infty$-representation of the $\infty$-groupoid of $M$. This
integration procedure can be understood as a consequence of Gugenheim's $\mathsf{A}_\infty$-version of de-Rham's theorem \cite{G}.

We consider flat connections with values in a fixed finite dimensional complex $(V,\partial)$,
i.e. $V$ is concentrated in finitely many degrees and each of its homogeneous components is finite dimensional.
The integration functor $\int$ restricts to a functor between the resulting full subcategories of $\mathsf{Rep}_\infty(TM)$ and $\mathsf{Rep}_\infty(\pi_\infty(M))$, denoted
by $\mathsf{Rep}_\infty(TM,V)$ and $\mathsf{Rep}_\infty(\pi_\infty(M),V)$, respectively.
Given a flat superconnection $\alpha$ with values in $(V,\partial)$, the integration functor
associates holonomies to any simplex in $M$. By construction,
the holonomy associated to a path $\gamma$ -- seen as a $1$-simplex -- coincides with ordinary
parallel transport along $\gamma$ and yields an automorphism of $(V,\partial)$.
The holonomy $\mathsf{Hol}(\sigma)$ of an $n$-dimensional simplex $\sigma$ is a linear endomorphism of $V$ of degree $1-n$.
The fact that holonomies are built coherently is formalized by
expressing the commutator $[\partial,\mathsf{Hol}(\sigma)]$ in terms of the holonomies
associated to subsimplices of $\sigma$.

For the purpose of this paper we want to focus on  $1$ and $2$-dimensional holonomies. To this end, one factors out all the information
related to simplices of dimension strictly larger than $2$. More formally, the flat superconnection
$\alpha$ is a sum of components of fixed form-degree and we disregard all components of form-degree
strictly larger than $2$. Similarly, one can truncate the $\infty$-groupoid $\pi_\infty(M)$ of $M$ which leads to the fundamental $2$-groupoid $\pi_{\le 2}(M)$ mentioned above.
Observe that, formally, we also ignore the categorical structure on $\mathsf{Rep}_\infty(TM,V)$ and $\mathsf{Rep}_\infty(\pi_\infty(M),V)$ respectively, since we work on the level of objects only. We indicate this transition with the change of notation
from $\mathsf{Rep}_\infty$ to $\mathrm{Rep}_\infty$.

As mentioned above, the holonomies associated to paths are elements of the automorphism group
of $(V,\partial)$. Hence one expects that the holonomies associated to $2$-simplices should 
belong to the automorphism $2$-group of $(V,\partial)$. One way to make this precise is to consider
the Lie crossed module $\mathsf{GL}(V)$ associated to $V$ and its infinitesimal version $\mathsf{gl}(V)$.
We show that the truncation of a flat superconnection can be seen as a flat $2$-connection
on $M$ with values in $\mathsf{gl}(V)$ and that the $2$-truncation of $\int \alpha \in \Rep_\infty(\pi_\infty(M),V)$
is a $2$-representation of $\pi_{\le 2}(M)$ on $(V,\partial)$.
Hence, we obtain a diagram of the form
$$ 
\xymatrix{
\mathrm{Rep}_\infty(TM,V) \ar[rr]^\int \ar[d]_{\mathsf{T}_{\le 2}} && \mathrm{Rep}_\infty(\pi_\infty(M),V) \ar[d]^{\mathsf{T}_{\le 2}} \\
\mathsf{Flat}(M,\mathsf{gl}(V)) && \mathrm{Rep}(\pi_{\le 2}(M),\mathsf{GL}(V)).
}
$$

This allows us to compare the integration functor $\int$ to the $2$-dimensional holonomies for flat $2$-connections with values in a differential Lie crossed module, which specializes to
$$ \mathsf{Hol}: \Flat(M,\mathsf{gl}(V)) \to \mathrm{Rep}(\pi_{\le 2}(M),\mathsf{GL}(V)).$$

Our main result is that the above two diagrams are compatible, i.e.
the $2$-truncation of $\int \alpha$ actually coincides with the holonomy construction
proposed in \cite{BaezS,Picken,SchreiberWaldorf}. More precisely:

\begin{theorem*}
Let $M$ be a smooth manifold and $(V,\partial)$ a cochain complex of vector spaces of finite type. Then the following diagram commutes:
\[
\xymatrix{
\mathrm{Rep}_\infty(TM,V)\ar[rr]^{\int} \ar[d]_{\mathsf{T_{\le 2}}} && \mathrm{Rep}_\infty(\pi_\infty(M),V)\ar[d]^{\mathsf{T_{\le 2}}}\\
\mathsf{Flat}(M,\mathsf{gl}(V)) \ar[rr]_{\mathsf{Hol}} &&  \mathrm{Rep}(\pi_{\le 2}(M),\mathsf{GL}(V)).}
\]
\end{theorem*}

\begin{conventions}
By `cochain complex' we mean a finite dimensional cochain complex of vector spaces over $\mathbb{R}$,
i.e. the differential increases the degree by $1$, the vector space is supported in finitely many degrees and each of its homogeneous components is finite-dimensional.

We denote the unit interval [0,1] by $I$. Moverover, in the square $I^2$, we denote the horizontal coordinate by $t$
and the vertical coordinate by $s$.

The $n$-simplex is the subset of $\mathbb{R}^n$ given by
$\Delta_n := \{ (t_1,\dots,t_n): 1\ge t_1 \ge t_2 \ge \cdots \ge t_n \ge 0\}$.
Given any map $f: I \times X \to Y$ we obtain maps
$$ f_{(k)}: \Delta_k \times X \to Y^{\times k}, \quad (t_1,\dots,t_n,x) \mapsto (f(t_1,x),\dots,f(t_n,x))$$
for any $k\ge 1$.
\end{conventions}

\subsection*{Acknowledgements}

We would like to thank Jim Stasheff for carefully reading and commenting a previous version of this manuscript.
We are also grateful to Joao Faria Martins for his help with some of the references.

\section{Higher holonomies for superconnections}

Here we briefly review the parallel transport for superconnections \footnote{Although the prefix \textit{super} usually denotes a $\mathbb{Z}/2\mathbb{Z}$ grading, our vector bundles are $\mathbb{Z}$-graded.} introduced by Igusa \cite{I} and studied further by Block-Smith \cite{BS} and 
Arias Abad-Sch\"atz \cite {AS}.
\subsection{Flat superconnections }\label{section:flat_superconnections}

Let $E$ be a $\mathbb{Z}$-graded vector bundle of finite rank over a manifold $M$, i.e. $E = \bigoplus_{k\in \mathbb{Z}} E^k$, where $E^k$ is a finite-rank vector bundle
over $M$  for every $k\in \mathbb{Z}$ and $E^k$ is required to be zero for almost all $k \in \mathbb{Z}$.
The space $\Omega(M,E)=\Gamma(\wedge T^*M \otimes E)$ of differential forms with values in $E$ is a $\mathbb{Z}$-graded module over the algebra $\Omega(M)$.

\begin{definition}
A superconnection on $E$ is a linear operator
\begin{align*}
 D: \Omega(M,E) \to \Omega(M,E),
\end{align*}
of degree one which satisfies the Leibniz rule
\begin{align*}
D(\alpha \wedge \omega) = d\alpha \wedge \omega + (-1)^{|\alpha|} \alpha \wedge D(\omega),
\end{align*}
for all homogeneous $\alpha \in \Omega(M)$ and $\omega \in \Omega(M,E)$.
A superconnection $D$ is flat if $D^2=0$.
\end{definition}

\begin{remark}
The flat superconnections on $M$ can be organized into a $dg$-category.
A morphism of degree $k$ between two flat superconnections $(E,D)$ and $(E',D')$ on $M$  is a degree $k$ morphism of
$\Omega(M)$-modules
\[\phi: \Omega(M,E)\to \Omega(M,E').\]
Notice that we do not require $\phi$ to be a chain map.
The space of morphisms 
\[\underline{\Hom}(E,E')= \oplus_{k\in \mathbb{Z}} \underline{\Hom}^{{k}}(E,E'),\]
is a cochain complex with differential
\[\Delta(\phi):=D' \circ \phi -(-1)^{|\phi|}\phi \circ D. \]
We will denote the resulting $dg$-category by $\mathsf{Rep}_\infty(TM)$.
\end{remark}

\begin{definition}
Let $(V,\partial)$ be a cochain complex and $M$ be a manifold.
A flat superconnection on $M$ with values in $(V, \partial)$ is a flat superconnection
on the trivial graded vector bundle $M\times V$ such that the following diagram commutes:
\[
\xymatrix{
\Omega^k(M, V^p)\ar[rr]^D \ar[rrd]_\partial &  &\bigoplus_{l\geq k} \Omega^l(M, V^{k+p-l+1})\ar[d]^\pi\\
&&\Omega^{k}(M, V^{p+1})
}
\]
\end{definition}

\begin{lemma}
Let $(V,\partial)$ be a complex and $M$ be a manifold.
A superconnection on $M$ with values in $(V, \partial)$ corresponds naturally to
a differential form
\[ \alpha = \alpha^1+ \alpha^2+ \alpha^3+ \dots,\]
where $\alpha^i \in \Omega^i(M, \End^{1-i}(V)) = \Omega^i(M)\otimes \End^{1-i}(V)$.
Moreover $\alpha$ is flat if and only if the equation
 \[[\partial, \alpha^n]+ d \alpha^{n-1}+ \sum_{i+j=n}\alpha^i \circ  \alpha^j=0\]
holds
for each $n\geq 1$. Here, the commutator $[\cdot,\cdot]$ operates only on the endomorphism part of $\alpha^n$, while 
$d$ is the covariant derivative assosiated to the trivial connection on $\End(V)$. The symbol $\circ$
denotes the multiplication on $\Omega(M)\otimes \End(V)$ given by the wedge-product of forms and composition of endomorphisms, respectively.
\end{lemma}

\subsection{Representation up to homotopy of simplicial sets}

Let $X_{\bullet}$ be a simplicial set with face and degeneracy maps denoted by
\[d_i: X_k\to X_{k-1} \quad \text{and} \quad s_i: X_k \to X_{k+1},\] 
respectively.  We will use the notation
\begin{eqnarray*}
\front_i&:=&(d_0)^{k-i}:X_k \to X_i,\\
\back_i&:=&d_{i+1} \circ \dots \circ d_{k} :X_k \to X_i,
\end{eqnarray*}
for the maps that send a simplex to its $i$-th back and front face.
The $i$-th vertex of a simplex $\sigma \in X_k$ 
will be denoted $v_i(\sigma)$, or simply $v_i$, when no confusion can arise.
In terms of the above operations, one can write
\[v_i=(\front_0\circ \back_{i})(\sigma). \]

Suppose that $E$ is a $\mathbb{Z}$-graded vector bundle over $X_0$, i.e. there is a $\mathbb{Z}$-graded vector space
$E_x$ associated to each vertex $x \in X_0$.
A cochain $F$ of degree $k$ on $X_{\bullet}$ with values in $E$ is a map
\begin{align*}
F: X_k \to E,
\end{align*}
such that $F_k(\sigma)\in E_{v_0(\sigma)}$. We denote by $C^k(X, E)$ the vector space of normalized cochains, i.e. those cochains which vanish on the images of $s_i$. The space of $E$-valued cochains is naturally a $\mathbb{Z}$-graded vector space.
In case the vector bundle is the trivial line bundle $\mathbb{R}$ we will write $C(X)$ instead of $C(X,\mathbb{R})$.
The space $C(X)$ is naturally a $\mathbb{Z}$-graded $dg$-algebra with the usual cup product and the simplicial differential defined by
\[\delta(\eta)(\sigma):=\sum_{i=0}^k(-1)^id^*_i(\eta)(\sigma),\]
for $\eta \in C^{k-1}(X)$. Given any $\mathbb{Z}$-graded vector bundle $E$ over $X_0$, the cup product gives the space $C(X,E)$ the structure of a right graded module over the algebra $C(X)$.

\begin{definition}
A representation up to homotopy of $X_{\bullet}$ consists of the following data:
\begin{enumerate}
\item A finite rank $\mathbb{Z}$-graded vector bundle $E$ over  $X_0$.
\item A linear map $D: C(X,E)\to C(X,E)$ of degree $1$ which is a derivation with respect to the $C(X)$-module structure and squares to zero.
\end{enumerate}
The cohomology of $X_\bullet$ with values in $E$, denoted $H(X,E)$, is the cohomology of the complex $(C(X,E),D)$.
\end{definition}

\begin{remark}
The representations up to homotopy of $X_{\bullet}$ form a  $dg$-category. 
Let $(E,D), (E',D')$ be two representations up to homotopy of $X_{\bullet}$.
A degree $k$ morphism $\phi \in \RHom^{k}(E,E')$ is a degree $k$ map of $C(X)$-modules
$\phi: C(X,E)\to C(X,E').$
The space of morphisms is naturally a $\mathbb{Z}$-graded vector space
\[\RHom(E,E')=\bigoplus_k \RHom^{k}(E,E')\]
with differential
\begin{eqnarray*}
\D:\RHom(E,E')&\to& \RHom(E,E')\\
\phi &\mapsto & D' \circ \phi - (-1)^{|\phi|} \phi \circ D.
\end{eqnarray*}
 
We denote the resulting $dg$-category by $\mathsf{Rep}_{\infty}(X_\bullet )$.
%\begin{remark}
%The category  $\mathsf{Rep}_\infty(X_\bullet )$ is functorial with respect to maps of simplicial sets. Namely, if $f: X_\bullet \to %Y_\bullet$ is a morphism of simplicial sets then there is a pull-back $dg$-functor:
%\begin{eqnarray*}
%f^*:\URRep(Y_\bullet ) & \to & \URRep(X_\bullet )\\
%E &\mapsto &f^*(E)
%\end{eqnarray*}
%In particular, this $dg$-functor induces a map in cohomology:
%\[f^*:H(Y,E)\to H(X,f^*E) .  \]
\end{remark}

\begin{definition}
Let $(V,\partial)$ be a cochain complex and $X_\bullet$ be a simplicial set.
A representation up to homotopy of $X_\bullet$ on $(V, \partial)$ is a representation up to homotopy of $X_\bullet$
on the trivial graded vector bundle with constant fibre $V$ such that the following diagram commutes:
\[
\xymatrix{
C^k(X, V^p)\ar[rr]^D \ar[rrd]_\partial &  &\bigoplus_{l\geq k} C^l(X, V^{k+p-l+1})\ar[d]^\pi\\
&&C^{k}(X, V^{p+1})
}
\]
\end{definition}

\begin{remark}\label{remstruc}
In \cite{AS}, an equivalent description of representations up to homotopy in terms of endomorphism-valued cochains
is given. As a special case, we obtain the following description of a representation up to homotopy of $X_\bullet$
on $(V,\partial)$.
Let $(\beta^n)_{n \geq 1}$ be a family of cochains of $X_\bullet$ with values in the graded algebra $\End(V)$, where $\beta^n$ is a $n$-cochain that assigns to $\sigma \in X_n$
a linear map
\[\beta^n(\sigma): V\to V\]
of degree $1-n$.
These operators are required to satisfy the following  equations for $n\geq 1$:
\begin{equation}\label{structure equations}
[\partial,\beta^n(\sigma)] = \sum_{j=1}^{n-1}(-1)^j\beta^{n-1}(d_j \sigma)  +
\sum_{j=1}^{n-1} (-1)^{j+1} (\beta^{j}\cup \beta^{k-j})(\sigma). 
\end{equation}
Here,
$\cup$ is the multiplication on $C(X)\otimes \End(V)$ given by the cup product of cochains and the composition of endomorphisms, respectively.
Moreover, the fact that we work with normalized cochains translates into the conditions
\[ \beta^1(s_0(\sigma))=\id  \quad \text{for} \quad \sigma \in X_0 \qquad \quad \textrm{and} \qquad \quad \beta^n(s_i(\sigma))=0 \quad \text{for} \quad \sigma \in X_{n-1} \textrm{ if } n>1. \]
These conditions in turn imply that the $1$-dimensional holonomies are compatible with time reversal:
\[ \beta^1(\sigma^{-1})= \beta^1(\sigma)^{-1}\text{ for any $1$-simplex $\sigma$}.\]
\end{remark}

%Notice that unitality implies that the $1$-dimensional holonomies are compatible with time reversal, i.e.
%$\beta^1(\sigma^{-1}) = \beta^1(\sigma)^{-1}$.

\subsection{Holonomies for superconnections}\label{subsection: hol for superconnections}

We discuss the notion of parallel transport for flat superconnections.
It generalizes the fact that a flat connection
on a vector bundle corresponds to a representation of the fundamental groupoid $\pi_1(M)$ of $M$.
We denote the simplicial set of smooth simplices in $M$ by $ \pi_{\infty}(M)$, that is
\[\pi_{\infty}(M)_k:=\mathcal{C}^{\infty}(\Delta_k,M).\]

\begin{remark}
We use the following definition of $\Delta_\bullet$ as a cosimplicial space:
The objects are $\Delta_k:=\{(t_1,\cdots,t_k)\in \mathbb{R}^k: 1 \ge t_1 \ge \cdots \ge t_k \ge 0\}$,
and the structure maps are
\begin{eqnarray*}
\partial_i: \Delta_k \rightarrow \Delta_{k+1}, \quad (t_1,\dots, t_k) \mapsto
\begin{cases}
(1,t_1,\dots, t_k) & \textrm{for } i=0,\\
(t_1,\dots, t_{i-1}, t_i, t_i, t_{i+1}, \dots ,t_k) & \textrm{for } 0<i<k+1,\\
(t_1,\dots, t_k, 0) & \textrm{for } i=k+1.
\end{cases}
\quad \textrm{ and}\\
\epsilon_i: \Delta_k \rightarrow \Delta_{k-1}, \quad (t_1,\dots, t_k) \mapsto (t_1,\dots ,t_{i-1}, \widehat{t_i}, t_{i+1}, \dots, t_k),
\quad \textrm{respectively.}
\end{eqnarray*}
Correspondingly, the face and degeneracy maps on $\pi_{\infty}(M)$ are given by
$d_i:=\partial_i^*$ and $s_i:=\epsilon_i^*$.
\end{remark}

One of the central results of \cite{AS,BS,I}  is:

\begin{theorem}\label{theorem:integration_everything}
There is an $\mathsf{A}_{\infty}$-functor 
\begin{align*}
\int: \mathsf{Rep}_\infty(TM) \to \mathsf{Rep}_\infty(\pi_{\infty}(M))
\end{align*}
from the $dg$-category of flat superconnections on $M$ and the $dg$-category of representations up to homotopy of $\pi_{\infty}(M)$.
\end{theorem}

\begin{definition}
Given a complex of vector spaces $(V,\partial)$, we will denote by $\mathsf{Rep}_\infty(TM,V)$ the full subcategory of flat superconnections on $M$ with values in $(V,\partial)$ and by $\mathsf{Rep}_{\infty}(\pi_\infty(M),V)$ the full subcategory of 
representations up to homotopy of $\pi_\infty(M)$ on $(V,\partial)$.
The integration functor $\int$ restricts to an $A_\infty$-functor
$$\int: \mathsf{Rep}_\infty(TM,V) \to \mathsf{Rep}_\infty(\pi_\infty(M),V).$$

We denote the underlying map between the sets of objects by
$$ \int: \mathrm{Rep}_\infty(TM,V) \to \mathrm{Rep}_\infty(\pi_\infty(M),V).$$
\end{definition}

\begin{remark}
We will not describe the details of the construction of the $\mathsf{A}_\infty$-functor $\int$ here, the interested reader can find them in
\cite{I,BS,AS}. However, we briefly discuss those parts of the construction that will be needed later on.
Let $\alpha = \alpha^1 + \alpha^2 + \cdots$ be a flat superconnection on $M$
with values in $(V,\partial)$.
The integration functor turns $\alpha$ into a $\infty$-representation
of $\pi_\infty(M)$, which we denote by $\beta = \beta^1 + \beta^2 + \cdots$.

We are interested in $\beta^1$ and $\beta^2$.
It turns out that $\beta^1(\gamma) : V\to V$ is just the parallel transport with respect to $\alpha^1$ along
the path $t\mapsto \gamma(1-t)$, i.e. the invere of the usual holonomy along $\gamma$.
The interested reader can consult Remark 4.13 of \cite{AS}, for instance.

Concerning $\beta^2$, we first observe that the structure equations reduce to
$$ [\partial, \beta^2(\sigma) ] = -\beta^1(d_1\sigma) + (\beta^1\cup \beta^1)(\sigma),$$
which tells us that $\beta^2(\sigma)$ is a homotopy between the chain maps
$\beta^1(t\mapsto \sigma(t,t))$ and $\beta^1(t\mapsto \sigma(t,0)) \beta^1(t\mapsto \sigma(1,t))$.
To construct $\beta^2(\sigma)$, we fix a way to fold the square $I^2$ onto the $2$-simplex
$\Delta_2=\{1\ge t \ge s \ge 1\}$ along the principal diagonal. 
Let $q: I^2 \to \Delta_2$ be the map $q(t,s) = (\max\{t,s\},s)$ and $\lambda: I^2 \to I^2$ be the map determined
by the property that $\lambda(-,s)$ is the piecewise linear path which runs
through the points $(s,1) \to (s,0) \to (0,0)$ and arrives at $(s,0)$ at time $t=1/2$.
We denote the composition $q\circ \lambda$ by $\Theta$ and set:
\begin{itemize}
 \item $\Theta_{(k)}$ is the map $\Delta_k\times I \to (\Delta_2)^{\times k}, (t_1,\cdots,t_k,s)\mapsto (\Theta(t_1,s),\cdots, \Theta(t_k,s))$,
 \item $p_i: (\Delta_2)^{\times k} \to \Delta_2$ is the projection on the $i$-th factor.
\end{itemize}

With this notation the operator $\beta^2(\sigma)$ is defined by the infinite sum
$$ 
\sum_{m,n\ge 0} (-1)^{m+n+1} \int_{\Delta_{m+n+1}\times I} \Theta_{(m+n+1)}^*(p_1^*\sigma^*\alpha^1 \wedge \cdots
\wedge p_m^*\sigma^*\alpha^1 \wedge p_{m+1}^*\sigma^*\alpha^2 \wedge p_{m+2}^*\sigma^*\alpha^1\wedge
\cdots \wedge p_{m+n+1}^*\sigma^*\alpha^1),$$
 which can be checked to be absolutely convergent.

\end{remark}

\begin{definition}
Let $M$ be a smooth manifold. A $n$-simplex $\sigma$ in $(\pi_\infty(M))_n$ is thin if the map $\sigma: \Delta_n \rightarrow M$ has rank $<n$. A representation up to homotopy $\beta$ of $\pi_\infty(M)$ is strongly unital at level $n$ if \[ \beta^n(\sigma)=0,\] for each thin $n$-simplex $\sigma$ in $(\pi_\infty(M))_n$.
We will say that $\beta$ is well behaved if it is strongly unital at level $2$.
\end{definition}

\begin{lemma}
Let $M$ be a smooth manifold. A well behaved representation up to homotopy $\beta$ of $\pi_\infty(M)$ is compatible with concatenation, i.e.
\[ \beta^{1}(\gamma \ast \sigma)= \beta^{1}(\gamma) \beta^{1}(\sigma)\quad \text{for any two composable $1$-simplices $\sigma$ and $\gamma$}.\]
Moreover, if $\alpha$ is a flat superconnection on $M$ then the corresponding representation up to homotopy $\int (\alpha)$ is well behaved.
\end{lemma}
\begin{proof}
For the first statement we consider the affine map $\pi: \Delta_2 \rightarrow \Delta_1=[0,1]$ that sends
$v_0$ to $0$, $v_1$ to $\frac{1}{2}$ and $v_2$ to $1$. Let $f:[0,1]\rightarrow M$ be the path $ \gamma \ast \sigma$ and consider the $2$-simplex $f \circ \pi$. Since $\beta$ is well behaved 
$\beta^2(f \circ \pi)=0$, and therefore:
\[  \beta^{1}(\gamma \ast \sigma)= \beta^{1}(\gamma) \beta^{1}(\sigma).\]
It remains to prove that any representation up to homotopy of the form $\int (\alpha)$ is well behaved.
That is, we need to prove that $\int(\alpha)^2(\sigma)=0$ for any thin $2$-simplex $\sigma$.
This follows from the explicit formula given in Remark 2.12: since $\sigma$ has rank one, the pull back of the two-form
$\alpha^2$ gives zero, hence all the terms vanish.
\end{proof}

\section{Surface holonomy}
\subsection{Lie Crossed modules}

We will briefly review the definitions and basic facts regarding Lie crossed modules, following closely Faria Martins-Picken \cite{FariaPicken2,FariaPicken,Picken}.

\begin{definition}
A Lie crossed module is a Lie group homomorphism $\delta: E \to G$ together with a left action $\rhd$ of $G$ on $E$ by Lie group automorphisms such that:
\begin{enumerate}
\item  For any $g\in G$ and $e \in E$:
\[\delta (g \rhd e)=g (\delta e) g^{-1}.\]
\item For any $e,f \in E$:
\[ (\delta e) \rhd f=e f e^{-1}.\] 
\end{enumerate} 
\end{definition}

The definition of a Lie crossed module might appear unmotivated. There is an alternative way to think about cross modules
that may be more enlightening: a Lie crossed module is the same as a strict Lie $2$-group.

\begin{definition}
A strict 2-group is a strict $2$-category with one object, in which all $1$-morphisms and $2$-morphisms are invertible.
\end{definition}

Given a Lie crossed module $\delta: E \to G$, one can define a strict two group $\mathcal{G}$ as follows. The $1$-morphisms of 
$\mathcal{G}$ are the elements of $G$. If $g$ and $h$ are $1$-morphisms, then a $2$-morphismfrom $h$ to $g$ is a diagram

\begin{equation}
 \xymatrix{ \ast& e&\ar@/^1pc/[ll]^{h}\ar@/_1pc/[ll]_{g}   \ast}
\end{equation}
where $e \in E$ is such that $\delta e = h^{-1}g$. There are horizontal and vertical composition operations for $2$-morphisms. 
The horizontal composition in $\mathcal{G}$ is always defined and is given by
$$  
\xymatrix{  \ast & e&\ar@/^1pc/[ll]^{h}\ar@/_1pc/[ll]_{g}    
\ast&e'& \ar@/^1pc/[ll]^{h'}\ar@/_1pc/[ll]_{g'} \ast &:=& \ast&  (h'^{-1}\rhd e)e' &\ar@/^1pc/[ll]^{hh'}\ar@/_1pc/[ll]_{gg'}   \ast}
$$
The vertical composition is only well defined if the lower edge of the first $2$-morphism matches the upper edge of the second. It
is given by the formula
$$ \xymatrix{  &\ast& e&\ar@/^1pc/[ll]^{h}\ar@/_1pc/[ll]_{g}   \ast\\ &&&&=&  \ast& e'e&\ar@/^1pc/[ll]^{k}\ar@/_1pc/[ll]_{g}   \ast
\\ &  \ast& e'&\ar@/^1pc/[ll]^{k}\ar@/_1pc/[ll]_{h}   \ast  }$$

The axioms for a Lie crossed module precisely guarantee that these operations define a strict $2$-group.
For example, the fact that the horizontal composition is well defined follows from condition $(1)$.
Given a diagram

$$ 
\xymatrix{  \ast & e&\ar@/^1pc/[ll]^{h}\ar@/_1pc/[ll]_{g}     \ast & e'&\ar@/^1pc/[ll]^{h'}\ar@/_1pc/[ll]_{g'} \\ \\
\ast & f&\ar@/^1pc/[ll]^{k}\ar@/_1pc/[ll]_{h}     \ast & f'&\ar@/^1pc/[ll]^{k'}\ar@/_1pc/[ll]_{h'} }
$$

there are, in principle, two ways to compose: first horizontaly and then vertically, or the other way around. The fact
that these two compositions coincide follows from condition $(2)$ in the definition of a crosse module.
\begin{definition}
A differential Lie crossed module is a Lie algebra homomorphism $\delta: \mathfrak{e} \to \mathfrak{g}$ together with a left action $\rhd$ of $\mathfrak{g}$ on $\mathfrak{e}$ by Lie algebra derivations such that:
\begin{enumerate}
\item  For any $X\in \mathfrak{g}$ and $v \in \mathfrak{e}$:
\[\delta (X \rhd v)=[X, \delta v].\]
\item For any $v,w \in \mathfrak{e}$:
\[ (\delta v) \rhd w=[v,w].\] 
\end{enumerate} 
\end{definition}

\begin{remark}
\hspace{0cm}
\begin{enumerate}
\item Lie crossed modules differentiate to differential Lie crossed modules. If $\delta: E \to G$ is a Lie crossed module, the corresponding Lie algebra homomorphism together with the corresponding Lie algebra action form a differential Lie crossed module. Hence differentiation yields a functor
from Lie crossed modules to differential Lie crossed modules. We call this functor the differentiation functor.

\item Differential Lie crossed modules can also be understood as incarnations of certain types of differential graded Lie algebras. Given a differential Lie crossed module
$\delta: \mathfrak{e}\to \mathfrak{g}$, one considers $\delta$ as a 2-term complex with $\mathfrak{e}$ in degree $-1$
and $\mathfrak{g}$ in degree $0$. The graded Lie bracket is then defined in terms of the Lie bracket on $\mathfrak{g}$ and the action of $\mathfrak{g}$ on
$\mathfrak{e}$.
\item In these notes we are interested in Lie crossed modules associated to complexes of vector spaces.
This can be done using the correspondence between special differential graded Lie algebras and differential crossed Lie modules mentioned above:
One first takes $\End(V)$ and considers it as a differential graded Lie algebra with the commutator bracket $[\cdot,\cdot]$ and the differential $[\partial,\cdot]$. Next, one reduces this to a $2$-term differential graded Lie algebra
by considering its $2$-truncation. This yields the following result:
\end{enumerate}
\end{remark}

\begin{lemma}[Faria Martins-Mikovic \cite{FariaMikovic}, Faria Martins-Picken \cite{FariaPicken}]\label{FMalg}
Let $(V, \partial)$ be a cochain complex. There is a differential Lie crossed module $\mathsf{gl}(V)= \big( \delta: \mathsf{gl}^{-1}(V) \to \mathsf{gl}^{0}(V) \big)$ defined as follows:
\begin{enumerate}
\item The Lie algebra $\mathsf{gl}^0(V)$ is the vector space of degree zero cochain maps $V \rightarrow V$, endowed with the commutator bracket.
\item The Lie algebra $\mathsf{gl}^{-1}(V)$ is the quotient Lie algebra 
\[ \mathsf{gl}^{-1}(V):= \frac{\End^{-1}(V)}{[\partial ,\End^{-2}(V))]},\]
where $\End^{-1}(V)$ is the space of degree $-1$ endomorphisms of $V$ endowed with the Lie bracket:
\begin{equation}\label{bracketend}
 [s,t]= s \partial t - t\partial s +s t \partial - ts \partial,
\end{equation}
and $[\partial ,\End^{-2}(V))]$ is the ideal of $ \End^{-1}(V)$ which consists of elements of the form $ \partial h - h \partial $.
\item The homomorphism $ \delta: \mathsf{gl}^{-1}(V) \to \mathsf{gl}^{0}(V)$ is given by 
\[ \delta(s) := \partial s + s \partial  .\]
\item The action $\rhd$ of $\mathsf{gl}^{0}(V)$ on $\mathsf{gl}^{-1}(V)$ is given by:
\[ \phi \rhd s:= \phi s - s \phi.\]
\end{enumerate}
\end{lemma}

The following lemma appeared in \cite{FariaPicken} (see also \cite{KP}). We reproduce a proof for the convenience of the reader:

\begin{lemma}
 Let $(V, \partial)$ be a cochain complex. Then there is a 
Lie crossed module \[\mathsf{GL}(V):= \big( \delta: \mathsf{GL}^{-1}(V) \to \mathsf{GL}^{0}(V \big),\] 
defined as follows:
\begin{enumerate}
\item The Lie group $\mathsf{GL}^0(V)$ is the group of automorphisms of the cochain complex $(V,\partial)$.
\item The Lie group $\mathsf{GL}^{-1}(V)$ is the quotient 
\[ \mathsf{GL}^{-1}(V):= \frac{\End^{-1}(V)'}{[\partial ,\End^{-2}(V)]}.\]
Here $\End^{-1}(V)'$ is the space of degree $-1$ endomorphisms $s$ of $V$ such that $[\partial,s]+\id$ is invertible. The group structure 
is given by the formula:
\[ s \ast t := s+ t + s\partial t + st \partial = s + t + s [\partial,t].\]
The space $[\partial, \End^{-2}]$ is a normal subgroup of $\End^{-1}(V)'$.

\item The homomorphism $\delta: \mathsf{GL}^{-1}(V) \to \mathsf{GL}^{0}(V)$ is given by 
\[ \delta(s) := [\partial,s]+ \id  .\]
\item The action $\rhd$ of $\mathsf{GL}^0(V)$ on $\mathsf{GL}^{-1}(V)$ is given by:
\[ \phi \rhd s:= \phi s \phi^{-1}.\]
\end{enumerate}
\end{lemma}
\begin{proof}
First observe that $\End^{-1}(V)'$ is open in $\End^{-1}(V)$. Clearly, the operation $\ast$ is smooth. Let us prove
that it gives  $\End^{-1}(V)'$ the structure of a Lie group. To check associativity we compute:
\begin{eqnarray*}
 (r\ast s )\ast t&=& (r+s + r\partial s + rs \partial) \ast t \\
&=&r+s  + r\partial s + rs \partial + t+ 
(r+s + r\partial s )\partial t+(r+s + r\partial s + rs \partial)t\partial\\
&=& r\ast (s\ast t).  
\end{eqnarray*}
Similarly, one easily checks that $\delta(s \ast t )= \delta(s)\delta(t)$, which implies that $\End^{-1}(V)'$ is closed
under the operation $\ast$.
It remains to prove the existence of inverses with respect to the operation $\ast$. Direct computation shows that:
\[ s \ast (- s\delta(s)^{-1})= (-s\delta(s)^{-1}) \ast s=0.\]

Since $[\partial, \End^{-2}(V)]$ is a closed subset of $\End^{-1}(V)$, it suffices to show that it is a normal subgroup in order to establish that it is a normal Lie subgroup.
To check that it is a subgroup we compute
\[ [\partial,x]\ast [\partial,y] = [\partial,x]+ [\partial,y]=[\partial,x+y].\]

Next, we need to prove that $[\partial, \End^{-2}(V)]$ is a normal subgroup. For this we observe that

\[ r^{-1} \ast [\partial,x] \ast r= [\partial, x \delta(r)]. \]

In order to prove that $\delta$ is a homomorphism we compute
\[ \delta (s \ast t)=\delta(s+t + s\partial t+ st\partial)=[\partial, s+t+s \partial t+ st\partial] +\id=([\partial,s]+\id)([\partial,t]+
\id)=\delta(s) \delta(t).
\]
Finally, we need to check that conditions $(1)$ and $(2)$ in the definition of a Lie crossed module are satisfied. For equation $(1)$
we compute:
\[ \delta(\phi \rhd t)= \delta( \phi t \phi^{-1})= \id + [\partial, \phi t \phi^{-1}]= \phi(\id + [\partial, t]) \phi^{-1}= \phi \rhd \delta(t).\]
For equation $(2)$ one needs to prove that
\[ \delta(r) s \delta(r)^{-1}= r \ast s \ast r^{-1},\]
which is equivalent to
\[ \delta(r) s = (r \ast s \ast r^{-1}) \delta(r).\]
We then compute the right hand side:

\begin{eqnarray*}
(r \ast s \ast r^{-1}) \delta(r)&=&(r^{-1} +(r\ast s)\delta(r^{-1})) \delta(r)= r^{-1}\delta(r)+ r +s +r\partial s + rs\partial\\
&=&-r+r +s +r\partial s + rs\partial= s+ r\partial s + rs\partial= \delta(r)s+[\partial,rs]= (\delta(r)s) \ast [\partial, rs].
\end{eqnarray*}
Since the equation only needs to hold in the quotient by $[\partial, \End^{-2}(V)]$, this concludes the proof.

\end{proof}

\begin{lemma}
  Let $(V, \partial)$ be a cochain complex. The differential Lie crossed module 
$\mathsf{gl}(V)$ is the result of applying the differentiation functor to the Lie crossed module $\mathsf{GL}(V)$. 
\end{lemma}

\begin{proof}
Clearly, the Lie algebra of the group $\mathsf{GL}^{0}(V)$ of automorphisms of $(V,\partial)$ is the Lie algebra $\mathsf{gl}^{0}(V)$ of endomorphisms
of $(V,\partial)$. Next, let us compute the tangent space at the identity of 
\[\mathsf{GL}^{-1}(V)=\frac{\End^{-1}(V)'}{[\partial ,\End^{-2}(V)]}.\]

Since $\End^{-1}(V)'$ is open in $\End^{-1}(V)$, and $[\partial ,\End^{-2}(V)]$ is a linear subspace, the tangent space at the 
identity is 
\[ \mathsf{gl}^{-1}(V):= \frac{\End^{-1}(V)}{[\partial ,\End^{-2}(V))]}.\]
In order to prove that the Lie bracket coincides with that on $\mathsf{gl}^{-1}(V)$ it suffices to prove that the bracket 
on the Lie algebra of the group $\End^{-1}(V)'$ is given by the formula in equation (\ref{bracketend}). To prove this, we observe that
the exponential map in the Lie algebra $\End^{-1}(V)'$ is given by the formula

\[ \exp(A)= \sum_{i,j \geq0} \frac{1}{(i+j+1)!}(A \partial )^{i}A (A \partial)^j.\]
With this formula at hand, one can compute the bracket:

\begin{eqnarray*}
 [A,B]&=&\frac{d}{dt}\bigg|_{t=0} \mathsf{Ad}_{\exp(tA)}(B)\\
&=& \frac{d}{dt}\bigg|_{t=0} \frac{d}{ds}\bigg|_{s=0} \exp(tA)\ast \exp(sB) \ast \exp(-tA)\\
&=& \frac{d}{ds}\bigg|_{s=0} \bigg( \frac{d}{dt}\bigg|_{t=0} \exp(tA)\ast \exp(sB) + \frac{d}{dt}\bigg|_{t=0}\exp(sB)\ast \exp(-tA)\bigg)\\
&=& \frac{d}{ds}\bigg|_{s=0} \bigg( A\partial \exp(sB)+ A \exp(sB) \partial - \exp(sB)\partial A -\exp(sB)A \partial\bigg)\\ 
&=& A \partial B + A B \partial - B \partial A - BA \partial.
\end{eqnarray*}
Next, we need to prove that the differential of the homomorphism $\delta: \mathsf{GL}^{-1}(V)\to \mathsf{GL}^{0}(V)$ is the homomorphism
$\delta: \mathsf{gl}^{-1}(V)^\to \mathsf{gl}^{0}(V)$. To check this we compute:

\[ \frac{d}{dt}\bigg|_{t=0} \delta(tA)= \frac{d}{dt}\bigg|_{t=0} (t[\partial,A]+ \id)=[\partial,A].\]

Finally, it remains to show that by differentiating the action of $\mathsf{GL}^0(V)$ on $\mathsf{GL}^{-1}(V)$, one obtains the corresponding
infinitesimal action. Let us denote by $\rhd'$ the infinitesimal action obtained
by differentiating the global one. Then, one computes:

\begin{eqnarray*}
A \rhd'B &=& \frac{d}{dt}\bigg|_{t=0} \frac{d}{ds}\bigg|_{s=0} \exp(tA)\rhd \exp(sB)\\
&=& \frac{d}{dt}\bigg|_{t=0} \frac{d}{ds}\bigg|_{s=0} \exp(tA) \exp(sB) \exp(-tA)\\
&=& \frac{d}{ds}\bigg|_{s=0} \bigg( \frac{d}{dt}\bigg|_{t=0} \exp(tA) \exp(sB)+ \frac{d}{dt}\bigg|_{t=0} \exp(sB) \exp(-tA) \bigg)\\
&=& \frac{d}{ds}\bigg|_{s=0} \bigg(A \exp(sB)- \exp(sB) A \bigg)= AB-BA.\\
\end{eqnarray*}
\end{proof}

\begin{remark}
Let $\delta: E\to G$ be a Lie crossed module. We will use the following structures:
 \begin{enumerate}
  \item The action of $G$ on $E$ differentiates to an action of $G$ on $\mathfrak{e}$, which we denote by $\rhd$.
  In the case of $\mathsf{GL}(V)$, this action is simply given by conjugation.
  \item If we fix an element $e\in E$, we have the left multiplication map 
  $$L_e: E\to E, \quad e' \mapsto e \cdot e'$$
  and its differential $(L_e)_*$ at $1\in E$, which is an isomorphism from $T_1E$ to $T_e E$.
   In the case of $\mathsf{GL}(V)$, this map reads
   $$ (L_e)_*(X) = X + e [\partial,X].$$
	Similarly, we have the right multiplication map $R_e$ and its differential $(R_e)_*$.
 \end{enumerate}

\end{remark}

\begin{definition}
Let $M$ be a smooth manifold and $ \delta: \mathfrak{e}\to \mathfrak{g}$ a differential Lie crossed module.
A local 2-connection is a pair of differential forms $(A,B)$, such that $ A \in \Omega^1(M, \mathfrak{g})$, $B\in \Omega^2(M,\mathfrak{e})$ and 
\[ \delta B = dA + \frac{1}{2} [A,A] =: F_A.\]
The curvature 3-form of a $2$-connection is the differential form
\[\mathcal{M}(A,B):=dB+ A\wedge^\rhd B \in \Omega^3(M,\mathfrak{e}).\]
A local 2-connection is said to be flat if its curvature $3$-form is zero.

\end{definition}

\begin{definition}
Given a smooth manifold $M$ and a differential crossed module $\mathsf{g}$, we denote by
$\mathsf{Flat}(M,\mathsf{g})$ the set of all flat local $2$-connections on $M$ with values in $\mathsf{g}$.
\end{definition}

\subsection{Holonomies with values in crossed modules}

Given a flat connection on a trivial vector bundle $M\times V$ over a manifold $M$, the holonomy construction gives a representation
of the fundamental groupoid of $M$ into $\mathsf{GL}(V)$. Analogously, given a crossed module  $\delta: E \to G$, one may be interested in morphisms from the fundamental $2$-groupoid of $M$ to the $2$-group $\mathcal{G}$ determined by $\delta: E\to G$. Since for our purposes it is not necessary to discuss the definition of the fundamental $2$-groupoid of a space, we will define directly what such a representation is.

\begin{definition}
A $2$-path in a manifold $M$ is a piecewise smooth map $\Gamma: I^2 \to M$ which is constant on the vertical sides. There are (nonassociative!) horizontal and vertical compositions defined for $2$-paths in an obvious way, provided the appropriate edges match. 
\end{definition}

\begin{remark}
Given a $2$-path $\Gamma: I^2 \to M$ we use the following notation
for the paths corresponding to the horizontal edges:
\[ \Gamma_1(t):=\Gamma(t,1); \,\,\ \Gamma_0(t):=\Gamma(t,0).\]
More generally, we write $\Gamma_s: I\to M$ for the path given by fixing the vertical coordinate to be $s\in I$, i.e.
$\Gamma_s(t):=\Gamma(t,s)$.
\end{remark}

\begin{definition}\label{2rep}
Let $M$ be a smooth manifold. A representation of the fundamental $2$-groupoid of $M$ on a Lie crossed module  $\delta: E \to G$ consists of the following data:
\begin{itemize}
\item For every piecewise smooth path $\gamma: I \to M$, there is an element $\mathsf{Hol}(\gamma)\in G$.
\item For every 2-path $\Gamma: I^2 \to M$, there is an element $\mathsf{Hol}(\Gamma) \in E$.
\end{itemize}
This assignment satisfied the following conditions:
\begin{enumerate}
\item If $\gamma$ is a constant path then $\mathsf{Hol}(\gamma)=1$.
\item If $\Gamma$ is a constant $2$-path then $ \mathsf{Hol}(\Gamma)=1$.
%\item If two paths $\gamma$ and $\gamma'$ are homotopic relative to the boundary then $\Hol(\gamma)=\Hol(\gamma')$.
\item If the $2$-paths $\Gamma$ and $\Gamma'$ are homotopic relative to the boundary then $\Hol(\Gamma)=\Hol(\Gamma')$.
\item If $\Gamma$ is a $2$-path then:
\[ \delta (\Hol(\Gamma))= \Hol(\Gamma_0)^{-1} \Hol(\Gamma_1).\]
In other words, $\Hol(\Gamma)$ is a two morphism in $\partial: E\to G$ 
from $\Hol(\Gamma_0)$ to $\Hol(\Gamma_1)$.
\item The assignment preserves composition of paths and vertical and horizontal composition of $2$-paths.
\end{enumerate}
\end{definition}

\begin{definition}
Given a Lie crossed module $\mathsf{G}$ and a manifold $M$, we denote by $\mathrm{Rep}(\pi_{\le 2}(M),\mathsf{G})$ the set of representations of the fundamental $2$-groupoid of $M$ in $\mathsf{G}$.
\end{definition}

Generalizing the case of ordinary flat principal bundles, representations with values in Lie crossed modules can be constructed 
from flat connections with values in the corresponding differential Lie crossed modules. The following theorem has been established in \cite{BaezS, FariaMikovic, FariaPicken2, Picken, SchreiberWaldorf}.

\begin{theorem}\label{theoremHol}
Let $\delta: E\to G$ be a Lie crossed module and suppose that $\alpha=(A,B)$ is a flat connection on $M$ with values in the differential Lie crossed module $\delta: \mathfrak{e} \to \mathfrak{g}$. Then there is a representation $\Hol_\alpha$ of the fundamental 
$2$-groupoid of $M$ in $\delta: E\to G$ defined as follows:
\begin{itemize}
\item The holonomy $\mathsf{Hol}_\alpha(\gamma) \in G$ associated to a path $\gamma$ is the usual holonomy associated to the connection $A$, i.e.  $\mathsf{Hol}_\alpha(\gamma)$ is equal to $g_\gamma(1)$, where $g_\gamma: [0,1]\to G$
is the solution to the differential equation
$$ \frac{d g_\gamma(t)}{d t} = - (R_{g_\gamma(t)})_* A\left(\frac{d \gamma}{dt} \right) $$
with boundary condition $g_\gamma(0) =1$.
\item The holonomy $\mathsf{Hol}_\alpha(\Gamma) \in E$ associated to a $2$-path $\Gamma$ is $h(1)$, where $h: [0,1]\to E$ is
the solution to the differential equation
\[ \frac{d h(s)}{d s} = (L_{h(s)})_* \left(\int_0^1 g_{\Gamma_s}(t)^{-1}\rhd B\left(\frac{\partial}{\partial t} \Gamma_s(t), \frac{\partial}{\partial s}\Gamma_s(t)\right) dt\right), \qquad h(0)=1,\]
and $(L_{e})_*$ is the differential at $1\in E$ of the map given by left multiplication
with $e$.

\end{itemize}
\end{theorem}

\begin{definition}\label{definitionHol}
Suppose $\mathsf{G}$ is a Lie crossed module with differential Lie crossed module $\mathsf{g}$. We denote by
\[ \mathsf{Hol}: \mathsf{Flat}(M,\mathsf{g}) = \{\text{flat connections on $M$ with values in } \mathsf{g} \} \to \mathrm{Rep}(\pi_{\le 2}(M),\mathsf{G})\]
the holonomy assignment from Theorem \ref{theoremHol}.
\end{definition}

\begin{remark}
The key observation concerning the definition of the holonomy of a $2$-path $\Gamma$ is that the variation of
$\mathsf{Hol}_\alpha(\Gamma_s): [0,1] \to G$ satisfies the differential equation:
$$  \frac{d \mathsf{Hol}_\alpha(\Gamma_s)}{d s} = (L_{\mathsf{Hol}(\Gamma_s)})_* \left(\int_0^1 \mathrm{Ad}_{g_{\Gamma_s}(t)^{-1}} F_A\left(\frac{\partial}{\partial t} \Gamma_s(t), \frac{\partial}{\partial s}\Gamma_s(t)\right) dt\right).$$
In the linear case, i.e. for $\mathsf{GL}(V)$, one can find an integral formula for 
$\Hol_\alpha(\Gamma)$. The following lemma appears in \S2.3.2 of \cite{FariaCirio} and also in the proof of Proposition $3.13$ of \cite{AC}.
\end{remark}

\begin{lemma}\label{lemma:2_hol_linear}
Let $\mathsf{GL}(V)$ be the Lie crossed module associated to a cochain complex $(V,\partial)$
and $\alpha$ a flat $2$-connection on $M$ with values in $\mathsf{gl}(V)$.
Then the $2$-holonomy of $\alpha$ associated to a $2$-path $\Gamma$ in $M$ is given by
$$ \Hol_\alpha(\Gamma) =  \left(  \int_0^1 \int_0^1 g_{\Gamma_s}(t)^{-1}\rhd B\left(\frac{\partial}{\partial t} \Gamma_s(t), \frac{\partial}{\partial s}\Gamma_s(t)\right)  \mathsf{Hol}_\alpha(\Gamma_s)^{-1} dt ds \right) \mathsf{Hol}_\alpha(\Gamma_1).$$
\end{lemma}

\begin{proof}
By definition, $\Hol(\Gamma)$ is the time-1-solution to the differential equation
\begin{eqnarray*}
 \frac{d h(s)}{d s} &=& (L_{h(s)})_* \left(\int_0^1 g_{\Gamma_s}(t)^{-1}\rhd B\left(\frac{\partial}{\partial t} \Gamma_s(t), \frac{\partial}{\partial s}\Gamma_s(t)\right) dt\right)\\
 &=& \int_0^1 g_{\Gamma_s}(t)^{-1}\rhd B\left(\frac{\partial}{\partial t} \Gamma_s(t), \frac{\partial}{\partial s}\Gamma_s(t)\right) dt \\
 & &  + h(s) \left(\int_0^1 \mathrm{Ad}_{g_{\Gamma_s}(t)^{-1}} F_A\left(\frac{\partial}{\partial t} \Gamma_s(t), \frac{\partial}{\partial s}\Gamma_s(t)\right) dt \right) \\
 &=& \int_0^1 g_{\Gamma_s}(t)^{-1}\rhd B\left(\frac{\partial}{\partial t} \Gamma_s(t), \frac{\partial}{\partial s}\Gamma_s(t)\right) dt + h(s) \mathsf{Hol}(\Gamma_s)^{-1} \frac{d \mathsf{Hol}_\alpha(\Gamma_s)}{ds}.
\end{eqnarray*}
From this it follows that $X(s):=h(s)\mathsf{Hol}(\Gamma_s)^{-1}$ satisfies the differential equation
$$
\frac{d X(s)}{ds} = \left(\int_0^1 g_{\Gamma_s}(t)^{-1}\rhd B\left(\frac{\partial}{\partial t} \Gamma_s(t), \frac{\partial}{\partial s}\Gamma_s(t)\right) dt \right) \mathsf{Hol}_\alpha(\Gamma_s)^{-1},
$$
which integrates to
$$
h(1) = \left( \int_0^1 \left(\int_0^1 g_{\Gamma_s}(t)^{-1}\rhd B\left(\frac{\partial}{\partial t} \Gamma_s(t), \frac{\partial}{\partial s}\Gamma_s(t)\right) dt \right) \mathsf{Hol}_\alpha(\Gamma_s)^{-1} ds \right) \mathsf{Hol}_\alpha(\Gamma_1).
$$

\end{proof}

\section{Relation between the two approaches}

\subsection{2-connections from superconnections}\label{subsection: conn = conn}

Here we compare flat superconnections with values in a cochain complex $V$ to connections with values in the 
corresponding associated differential Lie crossed module $\mathsf{gl}(V)$.

%\begin{remark}
%Recall from \S \ref{subsection: hol for superconnections} that we denote the set of flat superconnections on a manifold $M$
%with values in a chain complex $(V,\partial)$ by $\mathrm{Rep}_\infty(TM,V)$.
%\end{remark}

\begin{definition}
Let $\alpha= \alpha^1 + \alpha^2 +\alpha^3+ \dots $ be a 
superconnection on $M$ with values in $(V,\partial)$, i.e. $\alpha^i \in \Omega^i(M,\End^{1-i}(V))$.
Let $\mathsf{T}_{\le 2}(\alpha) = (A,B)$ be the pair of forms
given by
$$ A:= \alpha^1, \quad B:= -\pi \circ \alpha^2,$$
where $\pi$ is the natural projection map $ \End^{-1}(V)\to \End^{-1}(V)/[\partial,\End^{-2}(V)] = \mathsf{gl}^{-1}(V)$.
\end{definition}

\begin{lemma}\label{lemmaT1}
Let $\alpha= \alpha^1 + \alpha^2 +\alpha^3+ \dots $ be a 
superconnection on $M$ with values in $(V,\partial)$ which is flat.
Then $\mathsf{T}_{\le 2}(\alpha)$ is
a flat local 2-connection with values in $\mathsf{gl}(V)$.

\end{lemma}

\begin{proof}
Since $\alpha$ is a flat superconnection, the differential forms $\alpha^i$ satisfy the equations:
\begin{equation}\label{eqstructure flatness}
 [\partial, \alpha^n]+ d \alpha^{n-1}+ \sum_{i+j=n}\alpha^i\circ  \alpha^j=0,
 \end{equation}
for each $n \geq 1$. In particular, for $n=1$ this equation implies that $ \alpha^1$ takes values in the Lie subalgebra of chain maps, i.e.
$\alpha \in \Omega^1(M, \mathsf{gl}^0(V))$. Next, we need to prove the equation:
\[ \partial B = dA + \frac{1}{2} [A,A],\]
which is equivalent to equation (\ref{eqstructure flatness}) for $n=2$.
In order to prove flatness we need to show that the following equation holds:

\[dB+ A\wedge^\rhd B=0.\]

This is a consequence of equation (\ref{eqstructure flatness}) for $n=3$ and the fact that we work modulo $[\partial,\End^{-2}(V)]$.

\end{proof}
The previous lemma provides an assignment

$$ \mathsf{T}_{\le 2}: \mathrm{Rep}_{\infty}(TM,V) \to \mathsf{Flat}(M,\mathsf{gl}(V)), \quad \alpha \mapsto \mathsf{T}_{\le 2}(\alpha).$$

\subsection{Recovering representations of $\pi_{\le 2}(M)$}\label{subsection: rep = rep}

In this section we establish a comparison result between representations up to homotopy of the simplicial set $\pi_\infty(M)$ of singular chains on $M$ and representations of the $2$-groupoid $\pi_{\le 2}(M)$  on crossed modules.

\begin{definition}
Let $a$ and $b$ be the following embeddings of the $2$-simplex $\Delta_2=\{1\ge t \ge s\ge 0\}$ into the square:
$$ a(t,s):=(s,t) \qquad \textrm{and} \qquad b(t,s):=(t,s).
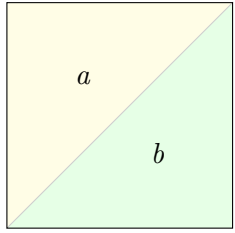
$$
\begin{center}
\begin{tikzpicture}[scale=0.5]
\draw (0,0) rectangle (6,6);
\draw[fill=yellow, opacity=0.1] (0,0)--(6,6)--(0,6)--cycle;
\draw[fill=green, opacity=0.1] (0,0)--(6,6)--(6,0)--cycle;
\node(a) at (2,4){a};
\node(b) at (4,2){b};
\end{tikzpicture}
\end{center}

Given a map $\Gamma: I^2 \to M$, we denote the compositions $\Gamma\circ a $ and $\Gamma \circ b$ 
by $\Gamma_{a}$ and $\Gamma_{b}$, respectively.
\end{definition}

\begin{remark}
Recall from \S \ref{subsection: hol for superconnections} that we denote the set of representations
up to homotopy of $\pi_{\infty}(M)$ on a chain complex $(V,\partial)$ by $\mathrm{Rep}_\infty(\pi_\infty(M),V)$.
Moreover, recall that for any $2$-path $\Gamma: I^2 \to M$, we defined
\[ \Gamma_s(t) := \Gamma(t,s).\]
We will also use the notation $\Gamma_d(t):= \Gamma(t,t)$.
\end{remark}

The following technical lemma will be useful in proving the main result of this section.

\begin{lemma}\label{relativehomotopy}
Let $M$ be a smooth manifold, $(V,\partial)$ a cochain complex and \[\beta=\beta^1+ \beta^2+\dots,\] a well behaved representation up to homotopy of $\pi_\infty(M)$ in $(V,\partial)$. Suppose that $\Gamma, \Gamma'$ are $2$-paths in $M$ which are homotopic relative to their boundaries. Then:
\[ \beta^2(\Gamma_a)-\beta^2(\Gamma_b) = \beta^2(\Gamma'_a)-\beta^2(\Gamma'_b) \,\,\,\mathrm{mod} [\partial,\mathsf{End}^{-2}(V)].\]
\end{lemma}
\begin{proof}
We will use Equation (\ref{structure equations}) in Remark \ref{remstruc} for $n=3$, which reads:
\begin{equation}\label{three simplex}
[\partial, \beta^3(\sigma)]=   \beta^{1}\cup \beta^{2}(\sigma)-\beta^2(d_1(\sigma))+\beta^{2}(d_2 \sigma)- \beta^{2}\cup \beta^{1}(\sigma) .
\end{equation}
This can be expressed graphically where a picture of a shaded simplex represents the holonomy that is assigned to it.
The equation above then reads:

\begin{equation*}
\begin{tikzpicture}[scale=1]
    % Define the heptagon coordinates
    \coordinate (A) at (-0.8,-0.3);
    \coordinate (B) at (0,-0.5);
    \coordinate (D) at (-0.1,0.5);
    \coordinate (C) at (0.5,-0.1);
     \coordinate (X) at (0,0.5);
    \coordinate (Y) at (-0.1,0.5);
    \coordinate (Z) at (-0.1,-0.5);
    \coordinate (W) at (0,-0.5);
     \coordinate (x) at (-0.5,0.5);
    \coordinate (y) at (-0.4,0.5);
    \coordinate (z) at (-0.4,-0.5);
     \coordinate (w) at (-0.5,-0.5);

\matrix[column sep=0.8cm,row sep=0.5cm]
{
\bleft & node{$\partial ,$} &\tetra &\bright & node{$=$}&\tetraA &node[$-$]& \tetraD & node{$+$}& \tetraB& node{$-$}& \tetraC&  node[$.$]\\
};
\end{tikzpicture}
\end{equation*}
Since the left hand side lies in $[\partial, \End^{-2}(V)]$, the right hand side is also zero in the quotient. The idea of the proof is very simple: Equation (\ref{three simplex}) gives relations between the holonomies associated to the faces of a tetrahedron. By triangulating the cube and using this equation for the simplices of the triangulation one obtains relations between holonomies associated to the faces of the cube.
We now considered the following picture which represents a cube seen from above with the vertices numbered.
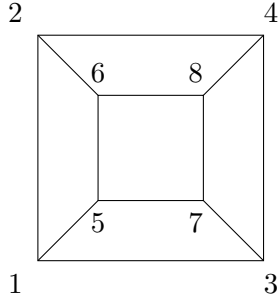
\begin{center}
%\begin{equation}
\begin{tikzpicture}[scale=1]
    % Define the heptagon coordinates
    \coordinate (a) at (0.7,0.7);
    \coordinate (b) at (0.7,-0.7);
    \coordinate (c) at (-0.7,0.7);
    \coordinate (d) at (-0.7,-0.7);
    \coordinate (A) at (1.5,1.5);
    \coordinate (B) at (1.5,-1.5);
    \coordinate (C) at (-1.5,1.5);
    \coordinate (D) at (-1.5,-1.5);

\node(1)  at(2,1.8){4};
\node (2) at (2,-1.8){3};
\node (3) at (-1.4,1.8){2};
\node (4) at (-1.4,-1.8){1};
\node (5) at (1,1){8};
\node (6) at (1,-1){7};
\node (7) at (-0.3,1){6};
\node (8) at (-0.3,-1){5};
\matrix[column sep=0.8cm,row sep=0.5cm]
{
&\cubefromabove\\
};
\end{tikzpicture}
\end{center}

Suppose that $H: [0,1]^2 \rightarrow M$ is a homotopy between $\Gamma$ and $\Gamma'$ relative to the boundary. If $i,j,k,l \in \{ 1,2,\dots,8\}$ we let $I_{ijkl}: \Delta_3 \rightarrow [0,1]^2$ be the unique affine map that sends $v_0$ to $i$, $v_1$ to $j$ and so on. We also write $F_{ijkl}$ for the composition $H \circ I_{ijkl}$.
Similarly, we write $I_{ijk}$ and $I_{ij}$ for the unique affine maps from $\Delta_2$ ( respectively $\Delta_1$) to $[0,1]^3$ that sends $v_0$ to $i$, $v_1$ to $j$ and $v_2$ to $k$. Finally, we set:
\[ T_{ijk}:=\beta^2(H \circ I_{ijk}) \quad \textrm{and} \quad L_{ij}=\beta^1(H \circ I_{ij}).\]

With this notation at hand, the proof amounts to applying Equation  (\ref{three simplex}) to the simplices of a triangulation of the cube, using the fact that some of the simplices are thin (because the homotopies are relative to the boundary) and combining the resulting equations. The following computations take place in the quotient space \[\frac{\End^{-1}(V)}{[\partial, \End^2(V)]}.\]

By applying Equation (\ref{three simplex}) to the $3$-simplex $F_{5783}$ one obtains:
\[ L_{57}T_{783} - T_{583} + T_{573} - T_{578}L_{83} = 0.\]
The tetrahedra $ I_{573}$ and $ I_{783}$ lie on the boundary of the cube and since the homotopy
is relative to the boundary,  $ H \circ I_{573}$ and $H\circ I_{783}$ are thin. Therefore $T_{573}$ and $T_{783}$ vanish. Moreover, $L_{83}$ is the identity because it is the holonomies of a constant path. We conclude that
\begin{equation}\label{eqlemma1}
T_{578} + T_{583}=0.
\end{equation}
holds.

The same analysis can applied to the other tetrahedra. We sum up to corresponding relations
between the $2$-holonomies below:

\begin{center}
 \begin{tabular}{c | c}
  tetrahedron & relation \\
  \hline
  $(5783)$ & $T_{583} + T_{578} = 0$\\
  $(5834)$ &  $T_{534} - T_{584}+T_{583}= 0$\\
  $(5134)$ &  $T_{134}-T_{534}+T_{514}=0$\\ 
  $(2568)$ &  $T_{568} + T_{258}=0$\\
  $(1248)$ &  $T_{148} - T_{128} + T_{124} = 0$\\
  $(1258)$ &  $T_{258} - T_{158} + T_{128} =0$ \\ 
  $(1584)$ &  $T_{584} - T_{184} + T_{154} - T_{158} = 0$\\
  $(5154)$ & $T_{154} + T_{514} = 0$\\
  $(1848)$ & $T_{148} + T_{184} = 0$
 \end{tabular}

\end{center}

Replacing Equations (5154) and (1848) into Equation (1584) one obtains: 
\begin{equation}\label{eqlemma10}
 T_{158}+T_{514}-T_{148}-T_{584}=0.
\end{equation}

Adding Equations (5834)) and (5134) and substracting Equation (5783) one obtains: 
\begin{equation}\label{eqlemma11}
  T_{578}-T_{134}=T_{514}-T_{584}.
\end{equation}

Adding Equations (1248) and (1258) and substracting Equation (2568) one obtains: 
\begin{equation}\label{eqlemma12}
  T_{124}-T_{568}=T_{158}-T_{148}.
\end{equation}
Adding Equations (\ref{eqlemma11}) and (\ref{eqlemma12}) one obtains:
\begin{equation}\label{eqlemma13}
 T_{578}-T_{134}+  T_{124}-T_{568}=T_{158}-T_{148}+T_{514}-T_{584}.
\end{equation}
The right hand side of Equation (\ref{eqlemma13}) vanishes in view of Equation (\ref{eqlemma10}).
On the other hand, the left hand side is precisely
\[\beta^2(\Gamma'_a)-\beta^2(\Gamma'_b)- \beta^2(\Gamma_a)+\beta^2(\Gamma_b).\]
This completes the proof.
\end{proof}
%%%%% Well behaved means: points mapto 1, inverse is inverse.
\begin{proposition}\label{propT2}
Let $M$ be a smooth manifold, $(V,\partial)$ a cochain complex and \[\beta=\beta^1+ \beta^2+\dots,\] a well behaved representation up to homotopy of $\pi_\infty(M)$ in $(V,\partial)$. There is a representation  $\mathsf{T}_{\le 2}(\beta)$ of the fundamental $2$-groupoid of $M$ in the crossed module $\mathsf{GL}(V)$ given by: 
\begin{itemize}
\item If $\gamma$ is a path in $M$, we set
\[ (\mathsf{T}_{\le 2}\beta)(\gamma):= \beta^1(\gamma^{-1}).\]
\item If $\Gamma$ is a $2$-path in $M$, we set
\[ (\mathsf{T}_{\le 2}\beta)(\Gamma):= (\beta^2(\Gamma_b)- \beta^2(\Gamma_a))\beta^1({\Gamma_1}^{-1}) .\]
\end{itemize}

\end{proposition}

\begin{proof}
Conditions $(1)$  and $(2)$ in Definition \ref{2rep} are direct consequences of the last statement in Remark \ref{remstruc}. The fact that condition $(3)$ is satisfied is precisely the content of Lemma \ref{relativehomotopy}.
Let us now prove that condition $(4)$ holds. That is, we need to prove that for any $2$-path $\Gamma$ the following equation holds:
\[ \delta (\mathsf{T}_{\le 2}\beta(\Gamma))= \mathsf{T}_{\le 2}\beta(\Gamma_0)^{-1} \mathsf{T}_{\le 2}\beta(\Gamma_1).\]

Specializing Equation (\ref{structure equations}) in Remark \ref{remstruc} to $n=2$ we obtain that:
\[ [\partial,\beta^2(\Gamma_{a})]=  \beta^1(\Gamma_1) -  \beta^1(\Gamma_d),\]
  
and also
\[ [\partial,\beta^2(\Gamma_{b})]= \beta^1(\Gamma_0) -  \beta^1(\Gamma_d). \]
Using these two equations, we compute:

\begin{eqnarray*}
 \delta (\mathsf{T}_{\le 2}\beta(\Gamma)) &=& [\partial,\beta^2(\Gamma_b)- \beta^2(\Gamma_a)] \beta^1 (\Gamma_1)^{-1}+\id\\
&=&(\beta^1(\Gamma_0) -  \beta^1(\Gamma_d)-\beta^1(\Gamma_1) +\beta^1(\Gamma_d))\beta^1(\Gamma_1)^{-1}+ \id\\
&=&\beta^1(\Gamma_0)\beta^1(\Gamma_1)^{-1}=\mathsf{T}_{\le 2}\beta(\Gamma_0)^{-1}\mathsf{T}_{\le 2}\beta(\Gamma_1).
\end{eqnarray*}

We need to prove that the assignment preserves composition. That it preserves composition of $1$-paths is a consequence of the fact that $\beta$ is well behaved.
So it remains to prove that the assignment preserves horizontal and vertical compositions of 
$2$-paths. In the rest of the proof we will use the same notation as in the proof of Lemma \ref{relativehomotopy}.
Let us begin with vertical composition. Consider two $2$-paths $\Gamma$ and $\Gamma'$ which are
vertically composable. We need to prove that if we subdivide the square horizontally as follows:

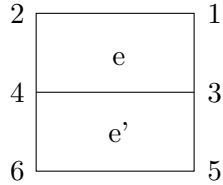
\begin{center}
%\begin{equation}
\begin{tikzpicture}[scale=1.5]
    % Define the heptagon coordinates
    \coordinate (a) at (-0.7,0.7);
    \coordinate (b) at (-0.7,0);
    \coordinate (c) at (-0.7,-0.7);
    \coordinate (d) at (0.7,0.7);
     \coordinate (e) at (0.7,0);
    \coordinate (f) at (0.7,-0.7);
    \coordinate (Z) at (-0.1,-0.5);
    
\node(1)  at(-0.6,0.7){2};
\node (2) at (-0.6,0){4};
\node (3) at (-0.6,-0.7){6};
\node (4) at (1.15,0.7){1};
\node (5) at (1.15,0){3};
\node (6) at (1.15,-0.7){5};
\node (7) at (0.3,-0.35){e'};
\node (8) at (0.3,0.3){e};

\matrix[column sep=0.8cm,row sep=0.5cm]
{
&\splitsquare \\
};
\end{tikzpicture}
\end{center}

We set
\[ e:=(T_{431}-T_{421})L_{12}, \quad e':=(T_{653}-T_{643})L_{34}, \quad e'':=(T_{651}-T_{621})L_{12},\]

and claim that
\[ e''=e'\ast e\]
holds.

By applying Equation  (\ref{three simplex}) to the $3$-simplex $F_{6421}$ and using that $\beta$ is well behaved and that the vertical sides of the square are constant, one obtains:
\begin{equation}\label{vertical1}
T_{421}-T_{621}+T_{641}=0.
\end{equation}

Similarly, applying Equation  (\ref{three simplex}) to the $3$-simplices $F_{6531}$ and $F_{6431}$ one obtains:
\begin{equation}\label{vertical2}
T_{631}-T_{651}+T_{653}=0 \qquad \textrm{and} \qquad T_{431}-T_{631}+T_{641}-T_{643}=0.
\end{equation}

These equations combine to

\begin{equation}\label{vertical4}
-T_{421}+T_{653}-T_{643}+T_{431}=T_{651}-T_{621}.
\end{equation}

We now use Equation (\ref{vertical4}) to compute:

\begin{eqnarray*}
e' \ast e &=&e' + e +e'[\partial, e]= e'+e+e'(\delta(e)-\mathrm{id})=e+ e'(L_{43} L_{12})\\
&=&(T_{431}-T_{421})L_{12}+(T_{653}-T_{643})L_{12}\\
&=&(T_{431}-T_{421}+T_{653}-T_{643})L_{12}=(T_{651}-T_{621})L_{12}=e''.
\end{eqnarray*}
This completes the proof that the assignment preserves vertical composition.

It only remains to prove that horizontal composition is preserved. 
Consider two $2$-paths $\Gamma$ and $\Gamma'$ which are
horizontally composable. We need to prove that if we subdivide the square horizontally as follows

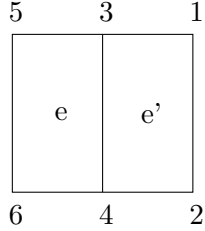
\begin{center}
%\begin{equation}
\begin{tikzpicture}[scale=1.5]
    % Define the heptagon coordinates
    \coordinate (a) at (-0.8,0.7);
    \coordinate (b) at (-0.8,-0.7);
    \coordinate (c) at (0,0.7);
    \coordinate (d) at (0,-0.7);
    \coordinate (e) at (0.8,0.7);
    \coordinate (f) at (0.8,-0.7);

    \node(1) at (-0.5,0.9){5};
    \node(2) at (-0.5,-0.9){6};
    \node(3) at (0.3, 0.9){3};
    \node(4) at (0.3,-0.9){4};
    \node(6) at (1.1,0.9){1};
    \node(7) at (1.1,-0.9){2};  

    \node(8) at (-0.1,0){e};
    \node(9) at (0.7,0){e'};

\matrix[column sep=0.8cm,row sep=0.5cm]
{
&\vsplitsquare \\
};
\end{tikzpicture}
\end{center}

and set
\[ e:=(T_{643}-T_{653})L_{35}, \quad e':=(T_{421}-T_{431})L_{13}, \quad e'':=(T_{621}-T_{651})L_{15}, \quad h:=L_{46},\]

then
\[ e''=(h^{-1}\rhd e' ) \ast e\]
holds.

As before, we evaluate Equation (\ref{three simplex}) on the $3$-simplices
$F_{6421}$, $F_{6531}$ and $F_{6431}$ and obtain the following relations:
\begin{eqnarray}
\label{hor1}
&L_{64}T_{421} -T_{621} + T_{641}=0,& \\
 \label{hor2}
&T_{631}-T_{651}+T_{653}L_{31}=0,& \\
\label{hor3}
&L_{64}T_{431}-T_{631}+T_{641}-T_{643}L_{31}=0.&
\end{eqnarray}

Substracting Equations (\ref{hor2}) and (\ref{hor3}) from Equation (\ref{hor1}) one obtains:

\begin{equation}\label{hor4}
T_{643}L_{31}-L_{64}T_{431}-T_{621}+L_{64}T_{421}-T_{653}L_{31}+T_{651}=0
\end{equation}
Multiplying Equation (\ref{hor4}) on the right by $L_{15}$ one obtains:
\begin{equation}\label{hor5}
T_{643}L_{35}-L_{64}T_{431}L_{15}+L_{64}T_{421}L_{15}-T_{653}L_{35}=(T_{621}-T_{651})L_{15}
\end{equation}
We now use Equation (\ref{hor5}) to compute:

\begin{eqnarray*}
(h^{-1}\rhd e' ) \ast e&=&(h^{-1}e' h)\ast e= (h^{-1}e' h) + e + (h^{-1}e' h)[\partial,e]\\
&=& (h^{-1}e' h) + e + (h^{-1}e' h)(\delta(e)-\id)=e+  (h^{-1}e' h)\delta(e) \\
&=&(T_{643}-T_{653})L_{35}+L_{64}(T_{421}-T_{431})L_{13}L_{46}L_{64}L_{35}\\
&=&(T_{643}-T_{653})L_{35}+L_{64}(T_{421}-T_{431})L_{15}=(T_{621}-T_{651})L_{15}=e''.
\end{eqnarray*}

\end{proof}

\begin{definition} Let $M$ be a manifold and $(V,\partial)$ be a cochain complex. By the previous proposition
we have an assignment 
\[ \mathsf{T}_{\le 2}:  \mathrm{Rep}_\infty(\pi_\infty(M),V) \to \mathrm{Rep}(\pi_{\le 2}(M),\mathsf{GL}(V)), \quad \beta \mapsto \mathsf{T}_{\le 2}(\beta).\]
\end{definition}

\subsection{Comparing the holonomies}

In this section we combine our previous considerations and prove that by truncating the integration functor $\int$ for 
representations up to homotopy one obtains the holonomy construction for the fundamental $2$-groupoid from Theorem \ref{theoremHol}.
More precisely, we have the following:

\begin{theorem}\label{main}
Let $M$ be a smooth manifold and $(V,\partial)$ a cochain complex. Then the diagram
\[
\xymatrix{
\mathrm{Rep}_\infty(TM,V)\ar[rrr]^{\mathsf{\int}} \ar[d]_{\mathsf{T}_{\le 2}} &&& \mathrm{Rep}_\infty(\pi_\infty(M),V)\ar[d]^{\mathsf{T}_{\le 2}}\\
\mathsf{Flat}(M,\mathsf{gl}(V)) \ar[rrr]_{\mathsf{Hol}} &&&  \mathrm{Rep}(\pi_{\le 2}(M),\mathsf{GL}(V))}
\]
is commutative,
where $\mathsf{T}_{\le 2}$ denotes the $2$-truncations from \S \ref{subsection: conn = conn} and \S\ref{subsection: rep = rep}, $\int$ is the $A_\infty$ integration-functor 
for representations up to homotopy from Theorem \ref{theorem:integration_everything} and $\mathsf{Hol}$ is the holonomy construction from Theorem \ref{theoremHol}.
\end{theorem}

Let $\alpha=\alpha^1+\alpha^2+\cdots$ be a flat superconnection on $M$ with values in $(V,\partial)$.
Denote the $\infty$-representation of $\pi_\infty(M)$ obtained by integrating $\alpha$
by $\beta=\beta^1 + \beta^2 + \cdots $.
Recall that the $2$-truncation $(A,B)$ of $\alpha$ is given by $A=\alpha^1$ and $B=-\pi\circ \alpha^2$,
where $\pi: \End^{-1}(V) \to \End^{-1}(V)/[\partial,\End^{-2}(V)]$ is the quotient map. 
\\

Let $\gamma$ be a path in $M$. By Remark 4.13. in \cite{AS}, $\beta^1(\gamma)$ equals the ordinary holonomy of $\alpha^1$ along the reversal of $\gamma$.
Hence $(\mathsf{T}_2\beta)(\gamma) = \beta^1(\gamma^{-1})$
is the ordinary holonomy of $\gamma$. By definition, this coincides with $\Hol_{\mathsf{T}_{\le 2}\alpha}(\gamma)$.
\\

Now let $\Gamma: I^2\to M$ be a $2$-path.
In Lemma \ref{lemma:2_hol_linear} we established the formula
$$\Hol_{\mathsf{T}_{\le 2}\alpha}(\Gamma) = \left(  \int_0^1 \int_0^1 g_{\Gamma_s}(t)^{-1}\rhd B\left(\frac{\partial}{\partial t} \Gamma_s(t), \frac{\partial}{\partial s}\Gamma_s(t)\right)  \mathsf{Hol}_\alpha(\Gamma_s)^{-1} dt ds \right) \mathsf{Hol}_\alpha(\Gamma_1).$$
for the $2$-holonomy of $(A,B)$ along $\Gamma$.
On the other hand, we defined
$$ (\mathsf{T}_{\le 2}\beta)(\Gamma) := (\beta^2(\Gamma_b) - \beta^2(\Gamma_a))\beta^1(\Gamma_1^{-1}).$$

Hence Theorem \ref{main} reduces to the identity
$$\beta^2(\Gamma_b) - \beta^2(\Gamma_a) =  \int_0^1 \int_0^1 g_{\Gamma_s}(t)^{-1}\rhd B\left(\frac{\partial}{\partial t} \Gamma_s(t), \frac{\partial}{\partial s}\Gamma_s(t)\right)  \mathsf{Hol}_\alpha(\Gamma_s)^{-1} dt ds. $$
To establish this, we start by rewriting the left-hand side in terms of iterated integrals.
We rely on the following standard facts about linear ODEs and iterated integrals:

\begin{lemma}
Let $\gamma: [0,1]\to M$ be a smooth path and $A \in \Omega(M,\mathfrak{gl}(V))$
the connection $1$-form of a trivial prinicipal $GL(V)$-bundle over $M$, where $V$ is a finite-dimensional vector space.

We define the holonomy $\Hol(\gamma)$ to be the time-1-solution to the differential equation
$$ \frac{d g_\gamma(t)}{dt} = - A_{\gamma(t)}\left(\frac{d \gamma(t)}{dt} \right) g_{\gamma}(t), \qquad g_\gamma(0)=1. $$ 
Let $a(t)$ be the $\mathsf{gl}(V)$-valued function on $[0,1]$ given by $\gamma^*A = a(t) dt.$
\begin{enumerate}
 \item The family $g_\gamma(t)$ can be represented by
 $$ g_\gamma(t) = \id + \sum_{n\ge 1}(-1)^n \int_{t\ge t_1\ge \cdots t_n \ge 0} a(t_1)  \cdots a(t_n) dt_1\cdots dt_n.$$
 \item The family $g_\gamma(t)^{-1}$ can be represented by
 $$ g_\gamma(t)^{-1} = \id + \sum_{n\ge 1}\int_{0 \le t_1 \le \cdots \le t_n \le t} a(t_1) \cdots a(t_n)dt_1\cdots dt_n.$$
 \item The family $g_{\gamma^{-1}}(t)$ can be represented by
 $$ g_{\gamma^{-1}}(t) = \id + \sum_{n\ge 1} \int_{t\ge t_1 \ge \cdots \ge t_n \ge 0} a(1-t_1)  \cdots a(1-t_n)dt_1\cdots dt_n.$$
\end{enumerate}
 
\end{lemma}

This immediately yields a representation of  $$\int_0^1 \int_0^1 g_{\Gamma_s}(t)^{-1}\rhd B\left(\frac{\partial}{\partial t} \Gamma_s(t), \frac{\partial}{\partial s}\Gamma_s(t)\right)  \mathsf{Hol}_\alpha(\Gamma_s)^{-1}dt ds $$
in terms of iterated integrals:

\begin{lemma}\label{lemmaintegral}
 Let $\Gamma: I^2 \to M$ be a $2$-paths and $(A,B)$ a $2$-connection on $M$ with values in $\mathsf{gl}(V)$,
 $V$ a cochain complex.
Define $a_s(t)$ to be the $\mathsf{gl}(V)^0$-valued function on $I^2$ given by $A_{\Gamma(t,s)}(\frac{\partial \Gamma}{\partial t})$. Moreover, let $b_s(t)$ be the $\mathsf{gl}(V)^{-1}$-valued function on $I^2$
given by $B_{\Gamma(t,s)}\left(\frac{\partial}{\partial t} \Gamma_s(t), \frac{\partial}{\partial s}\Gamma_s(t) \right)$.

The integral
$$ \int_0^1 \int_0^1 g_{\Gamma_s}(t)^{-1}\rhd B\left(\frac{\partial}{\partial t} \Gamma_s(t), \frac{\partial}{\partial s}\Gamma_s(t) \right)  \mathsf{Hol}_\alpha(\Gamma_s)^{-1} dt ds$$
can be represented by 
$$
Z(\Gamma):=\sum_{m,n\ge 0} \int_{\Delta_{m+n+1} \times I} a_s(1-t_1)\cdots a_s(1-t_m) b_s(1-t_{m+1}) a_s(1-t_{m+2}) \cdots a_s(1-t_{m+n+1}) dt_1 \cdots dt_{m+n+1} ds.
$$
\end{lemma}

\begin{proof}
As a straightforward consequence of the previous lemma,
we obtain that
$$ \int_0^1 \int_0^1 g_{\Gamma_s}(t)^{-1}\rhd B\left(\frac{\partial}{\partial t} \Gamma_s(t), \frac{\partial}{\partial s}\Gamma_s(t)\right)  \mathsf{Hol}_\alpha(\Gamma_s)^{-1} dt ds$$
can be represented by 
$$
\sum_{m,n\ge 0} \int_{\tilde{\Delta}_{m+n+1} \times I} a_s(t_1)\cdots a_s(t_m) b_s(t_{m+1}) a_s(t_{m+2}) \cdots a_s(t_{m+n+1}) dt_1 \cdots dt_{m+n+1} ds.
$$
where $\tilde{\Delta}_k$ denotes the {\em standard} simplex of dimension $k$, i.e.
$\tilde{\Delta}_k := \{0\le t_1\le \cdots \le t_k \le 1\}$.
Then one applies the diffeomorphism
$$ \Delta_k \to \tilde{\Delta}_k, \quad (t_1,\dots,t_k)\mapsto (1-t_1,\dots,1-t_k).$$
\end{proof}

Applying Lemma \ref{lemmaintegral} to a 2-connection, which is the $2$-truncation of a flat superconnection $\alpha$, yields

\begin{proposition}
Let $\alpha$ be a flat superconnection on $M$ with values in $(V,\partial)$.
Denote the corresponding flat $2$-connection with values in the differential Lie crossed module $\mathsf{GL}(V)$
by $(A,B)$. The integral
$$\int_0^1 \int_0^1 g_{\Gamma_s}(t)^{-1}\rhd B\left(\frac{\partial}{\partial t} \Gamma_s(t), \frac{\partial}{\partial s}\Gamma_s(t)\right)  \mathsf{Hol}_\alpha(\Gamma_s)^{-1} dt ds$$
can be written as
$$\sum_{m,n\ge 0}(-1)^{m+n} \int_{\Delta_{m+n+1} \times I} \mu_{(m+n+1)}^*( p_1^*\Gamma^*\alpha^1
\wedge \cdots \wedge p_m^*\Gamma^*\alpha^1 \wedge p_{m+1}^*\Gamma^*\alpha^2 \wedge p_{m+2}^*\Gamma^*\alpha^1 \wedge \cdots \wedge p_{m+n+1}^*\Gamma^*\alpha^1),$$
where $\mu_{(k)}$ is the map
$$ \mu_{(k)}: \Delta_k \times I \to (I^2)^{\times k}, \quad (t_1,\dots,t_k,s) \mapsto ((1-t_1,s),\dots,(1-t_k,s)).$$
\end{proposition}

Recall from Subsection \ref{subsection: hol for superconnections} that $\beta^2(\sigma)$ of a $2$-simplex $\sigma$ is given by
$$\sum_{m,n\ge 0} (-1)^{m+n+1} \int_{\Delta_{m+n+1}\times I} \Theta_{(m+n+1)}^*(p_1^*\sigma^*\alpha^1 \wedge \cdots
\wedge p_m^*\sigma^*\alpha^1 \wedge p_{m+1}^*\sigma^*\alpha^2 \wedge p_{m+2}^*\sigma^*\alpha^1\wedge
\cdots \wedge p_{m+n+1}^*\sigma^*\alpha^1).$$
We want to apply this to $\Gamma_b$ and $\Gamma_a$, where
$\Gamma$ is a $2$-path and $a$ and $b$ denote the maps from the $2$-simplex $\Delta_2$ to the square $I^2$
given by $b(t,s)=(t,s)$ and $a(t,s)=(s,t)$.

To simplify the computations, we make use of the homotopy invariance of $\Hol_{(A,B)}(\Gamma)$
and $\mathsf{T}_{\le 2}(\beta)$, see Theorem \ref{theoremHol} and Proposition \ref{propT2}, respectively.
We replace $\Gamma$ by a $2$-path which is obtained by shrinking the square and moving it into its lower right quarter.
In more detail, we use an approprate isotopy $\phi_\tau$ of $\mathbb{R}^2$ to push $I^2$ into its own interior
in such a way that, at time $\tau=1$, $\phi_1(I^2)$ is contained in the (interior of the) lower right quarter of $I^2$.
Denote the image of $I^2$ under $\phi_\tau$ by $Q_\tau$.
For each $\tau$, replace $\Gamma$ with $\tilde{\Gamma}_\tau$,
where $\tilde{\Gamma}_\tau := \Gamma \circ \phi_\tau^{-1}$ on $Q_\tau$ and extend $\tilde{\Gamma}_\tau$
to all of $I^2$ be letting it be constant along the flow lines from $\partial I^2$ to $\partial Q_\tau$.
It is straight forward to check that this can be achieved such that $\tilde{\Gamma}_\tau$ is a homotopy
relative endpoints.\footnote{One can further smooth $\tilde{\Gamma}_\tau$ so as to stay within the smooth category.}
A schematic picture of the situation looks like this:

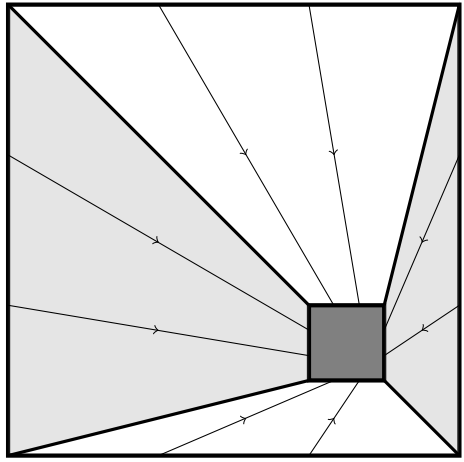
\tikzset{->-/.style={decoration={
  markings,
  mark=at position 0.5 with {\arrow{>}}},postaction={decorate}}}

\begin{center}
%\begin{equation}
\begin{tikzpicture}[scale=1]
\draw[ultra thick] (0,0) rectangle (6,6);
\draw[ultra thick,fill=gray] (4,1) rectangle (5,2);
\draw[thin,fill=gray, opacity=0.2] (0,6) -- (4,2) -- (4,1) -- (0,0) -- cycle (5,2) -- (6,6) -- (6,0) -- (5,1) -- cycle;
\draw[very thick] (0,6) -- (4,2) (0,0) -- (4,1) (5,2) -- (6,6) (5,1) -- (6,0);
\draw[->-] (0,4) -- (4,1.67);
\draw[->-] (0,2) -- (4,1.33);
\draw[->-] ((6,2)--(5,1.33);
\draw[->-] (6,4)--(5,1.67);
\draw[->-] (2,6) -- (4.33,2);
\draw[->-] (4,6) -- (4.67,2);
\draw[->-] (2,0) -- (4.33,1);
\draw[->-] (4,0) -- (4.67,1);
\end{tikzpicture}
\end{center}

We now redefine $\Gamma$ to be the $2$-path $\tilde{\Gamma}_1$. Observe that for this $2$-path, the rank has the following properties:
\begin{itemize}
\item it can be $2$ only inside the small square,
\item it is zero inside the shaded regions to the left and to the right of the small square.
\end{itemize}
We call a $2$-path with such a behaviour {\em well-supported}.
It easily follows that $\beta^2(\Gamma_a)$ vanishes for every well-supported $2$-path.

\begin{proposition}
Let $\Gamma$ be a well-supported $2$-path in $M$ and $\alpha = \alpha^1 + \alpha^2 + \cdots$ a flat superconnection
on $M$ with values in $(V,\partial)$.
The two elements of $\mathsf{GL}^{-1}(V)$ represented by
$$ Z(\Gamma) = \sum_{m,n\ge 0} (-1)^{m+n}\int_{\Delta_{m+n+1} \times I} \mu_{(m+n+1)}^{*}\Omega_{m,n}$$
and 
$$ X(\Gamma) = \sum_{m,n\ge 0} (-1)^{m+n+1}\int_{\Delta_{m+n+1} \times I} \Theta_{(m+n+1)}^*(b^{\times {m+n+1}})^* \Omega_{m,n}$$
coincide, where
$$ \Omega_{m,n} = p_1^*\Gamma^*\alpha^1 \wedge \cdots
\wedge p_m^*\Gamma^*\alpha^1 \wedge p_{m+1}^*\Gamma^*\alpha^2 \wedge p_{m+2}^*\Gamma^*\alpha^2\wedge
\cdots \wedge p_{m+n+1}^*\Gamma^*\alpha^1.$$
\end{proposition}

\begin{proof}

\hspace{0cm}
{\bf Step 1:} {\em replace $I^2$ by $\Delta_2$}

Recall from \S \ref{subsection: hol for superconnections} that $\Theta: I_2 \to \Delta^2$ is the composition $\Theta = q\circ \lambda$, with $\lambda: I^2\to I^2$ some 
map defined via a one-parameter family of piecewise linear paths inside $I^2$ and $q: I^2\to \Delta_2$ the projection map.
We define $r:=b\circ q$ to be the map that first collapses the square on the $2$-simplex, and then embeds the $2$-simplex into $I^2$.
By definition $b\circ \Theta = r\circ \lambda$. Moreover, we claim that $Z(\Gamma)$ does not change
if we replace $\mu_{(m+n+1)}^*\Omega_{m,n}$ with $\mu_{(m+n+1)}^*(r^{\times m+n+1})^*\Omega_{m,n}$.
This is a consquence of the fact that the pull back of $\Omega_{m,n}$ along $r^{\times m+n+1}$ coincides with
$\Omega_{m,n}$ on $\Delta_2^{m+n+1} \subset (I^2)^{\times m+n+1}$.
And since $\Gamma$ is well-supported, only the preimage of $(\Delta_2)^{\times m+n+1}$ under $\mu_{(m+n+1)}$ contributes to the integral.

To sum up, we can rewrite $Z(\Gamma)$ and $X(\Gamma)$ as follows:
\begin{eqnarray*}
Z(\Gamma) &=& \sum_{m,n\ge 0} (-1)^{m+n}\int_{\Delta_{m+n+1} \times I} \mu_{(m+n+1)}^{*}\tilde{\Omega}_{m,n},\\
X(\Gamma) &=& \sum_{m,n\ge 0} (-1)^{m+n+1} \int_{\Delta_{m+n+1}\times I} \lambda_{(m+n+1)}^*\tilde{\Omega}_{m,n},
\end{eqnarray*}  
where $\tilde{\Omega}_{m,n} = (r^{\times m+n+1})^*\Omega_{m,n}$.
\\

{\bf Step 2:} {\em find a homotopy}

We claim that there is a homotopy $H: I^2\times I \to I^2$ relative to the boundary $\partial I^2$
between $\mu$ and $\lambda \circ f$, where $f$ is the flip with respect to the coordinate $s$, i.e.
$$ f: I^2 \to I^2, \quad (t,s) \mapsto (t,1-s).$$
Recall that $\mu$ is just given by $(t,s)\mapsto (1-t,s)$, while $\lambda(-,s)$ is the piecewise
linear path which runs from $(s,1)$ to $(s,0)$, and finally to $(0,0)$.

We first homotop $\lambda$ to $\tilde{\lambda}(t,s) = (s,1-t)$. This can be achieved by running along $\lambda(-,s)$,
but only up to some time $\tau$ (and the rescaling of the time variable $t$). Hence $\lambda \circ f$
is homotopic to $(s,1-t) \mapsto (1-s,1-t)$.

After a change of variables such that the square is centered at the origin, the problem is to find a homotopy between the identity
and $(t,s) \mapsto (-s,t)$, seen as maps from $[-1/2,1/2]\times [-1/2,1/2]$ to itself. Clearly, this can be done -- for instance, one can use the homoeomorphism between the
square and the disk, and apply a rotation of $\pi/2$ to the latter. For the sake of concreteness, and 
for the better understanding of the regularity of the obtained homotopy, we choose to work directly on the square:
\begin{itemize}
\item We define two subsets of $[-1/2,1/2]\times [-1/2,1/2]$ as follows:
\begin{eqnarray*}
A &:=&\{(x,y): x>0, |y|<|x|\} \cup \{(x,-x): x\ge 0\},\\
B &:=& \{(x,y): y>0, |y|>|x|\} \cup \{(x,x): x\ge 0\}.
\end{eqnarray*}
Observe that $A$, $B$, $-A$ and $-B$ cover the square and only intersect in $(0,0)$.
%\begin{equation}
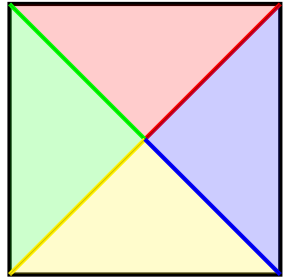
\begin{center}
\begin{tikzpicture}[scale=0.6]
\draw[ultra thick] (0,0) rectangle (6,6);
\draw[ultra thick,color=red] (3,3) -- (6,6);
\draw[fill=red, opacity=0.2] (3,3)--(6,6)--(0,6)--cycle;
\draw[ultra thick, color=green]  (3,3) -- (0,6);
\draw[fill=green, opacity=0.2] (3,3)--(0,6)--(0,0)--cycle;
\draw[ultra thick,color=yellow]  (3,3)--(0,0);
\draw[fill=yellow, opacity=0.2] (3,3)--(0,0)--(6,0)--cycle;
\draw[ultra thick, color=blue]  (3,3)--(6,0);
\draw[fill=blue, opacity=0.2] (3,3)--(6,0)--(6,6)--cycle;
\end{tikzpicture}
\end{center}
\item We define $X$ to be the vector field on the square given by
$$ X= \begin{cases} +2x\frac{\partial}{\partial y} & \textrm{on } A\cup -A\\
-2y \frac{\partial}{\partial x} & \textrm{on }B \cup -B.
\end{cases}
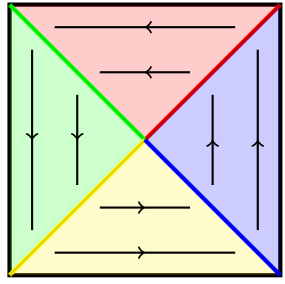
$$
\begin{center}
\begin{tikzpicture}[scale=0.6]
\draw[ultra thick] (0,0) rectangle (6,6);
\draw[ultra thick,color=red] (3,3) -- (6,6);
\draw[fill=red, opacity=0.2] (3,3)--(6,6)--(0,6)--cycle;
\draw[ultra thick, color=green]  (3,3) -- (0,6);
\draw[fill=green, opacity=0.2] (3,3)--(0,6)--(0,0)--cycle;
\draw[ultra thick,color=yellow]  (3,3)--(0,0);
\draw[fill=yellow, opacity=0.2] (3,3)--(0,0)--(6,0)--cycle;
\draw[ultra thick, color=blue]  (3,3)--(6,0);
\draw[fill=blue, opacity=0.2] (3,3)--(6,0)--(6,6)--cycle;
\draw[thick, ->-] (4.5,2)--(4.5,4);
\draw[thick, ->-] (5.5,1) --(5.5,5);
\draw[thick, ->-] (4,4.5)--(2,4.5);
\draw[thick, ->-] (5,5.5) --(1,5.5);
\draw[thick, ->-] (1.5,4)--(1.5,2);
\draw[thick, ->-] (0.5,5) --(0.5,1);
\draw[thick, ->-] (1,0.5)--(5,0.5);
\draw[thick, ->-] (2,1.5) --(4,1.5);
\end{tikzpicture}
\end{center}
Notice that $X$ is smooth when it is restricted to any of the subsets $A$, $B$, $-A$ or $-B$, but is not even continuous along the diagonals.

\item The homotopy $G$ between the identity and the map $(t,s)\mapsto (-s,t)$ is defined to be the time-$\tau$-flow along $X$.
This is indeed a continuous map, which is continuously differentiable on a dense and open subset of $I^2 \times I$.
 Moreover,
the derivatives of $G$ are bounded. Finally, for each fixed $\tau$, $G_\tau$ is a self-homeomorphism of the square
and it is easy to check that $G_0 =\mathrm{id}$ and $G_1(t,s)=(-s,t)$ holds.
For later, we also note that
the pull back of a differential form $\alpha$ along $G$ is well-defined on a dense and open subset
and the integral of $G^*\alpha$ is a well-defined real number.

To sum up, we proved existence of an appropriate homotopy $H$ between $\mu$ and $\lambda\circ f$
where $f$ is the flip $(t,s) \mapsto (t,1-s)$.
\end{itemize}

{\bf Step 3:} {\em apply Stokes Theorem}

The homotopy $H$ from Step 2 yields a homotopy
\begin{eqnarray*}
H_{(m+n+1)}: \Delta_{m+n+1}\times I \times I &\to & (I^{2})^{\times m+n+1}, \\
((t_1,\dots,t_{m+n+1}),s,\tau) & \mapsto & (H_\tau(t_1,s),\dots,H_\tau(t_{m+n+1},s))
\end{eqnarray*}
between $\mu_{(m+n+1)}$ and $(\lambda\circ f)_{(m+n+1)}$.
We observe that the last map equals 
$$\lambda_{(m+n+1)} \circ (\mathrm{id}\times f): \Delta_{m+n+1}\times I \to (I^2)^{\times m+n+1}.$$
As a consequence, the integral of the pull back of a differential form along $(\lambda\circ f)_{(m+n+1)}$ equals $-1$
the integral of the pull back along $\lambda_{(m+n+1)}$.

After these preparations, we are in position to prove that, modulo $[\partial,\End^{-2}(V)] \subset \End^{-1}(V)$,
the elements
$$ 
Z(\Gamma) = \sum_{m,n\ge 0} (-1)^{m+n}\int_{\Delta_{m+n+1} \times I} \mu_{(m+n+1)}^{*}\tilde{\Omega}_{m,n}
$$
and
\begin{eqnarray*}
X(\Gamma) &=& \sum_{m,n\ge 0} (-1)^{m+n+1} \int_{\Delta_{m+n+1}\times I} \lambda_{(m+n+1)}^*\tilde{\Omega}_{m,n}\\
&=& \sum_{m,n\ge 0} (-1)^{m+n} \int_{\Delta_{m+n+1}\times I} (\lambda \circ f)_{(m+n+1)}^*\tilde{\Omega}_{m,n}
\end{eqnarray*}  
are equal.
To this end we apply Stokes Theorem to the differential forms
$H_{(m+n+1)}^*(d\tilde{\Omega}_{m,n})$, i.e. we evaluate the equation
$$ \sum_{m,n\ge 0} (-1)^{m+n}\int_{\Delta_{m+n+1}\times I \times I} H^*_{(m+n+1)}(d\tilde{\Omega}_{m,n}) =  \sum_{m,n\ge 0} (-1)^{m+n} \int_{\partial(\Delta_{m+n+1}\times I \times I)} H^*_{(m+n)}\tilde{\Omega}_{m,n}.$$

Let us first look at the terms arising at the left hand side of the equation:
\begin{itemize}
\item[(D1)] Suppose $d$ is applied to a factor with $\alpha^1$. We replace this factor using the relation $d\alpha^1 + \alpha^1 \wedge \alpha^1 + [\partial, \alpha^2]=0$.
\item[(D2)] If we apply $d$ to the factor with $\alpha^2$, we get no contribution, since $d\alpha^2$ is a $3$-form, which we pull back by a map from the square.
\end{itemize}

We now look at the codimension one boundary strata of $\partial(\Delta_{m+n+1}\times I \times I)$:
\begin{itemize}
\item[($\partial$1)] There are the boundaries of the second interval, which correspond to setting the parameter $\tau$ to either $0$ or $1$. These two boundary strata just yield the difference between $Z(\Gamma)$ and $X(\Gamma)$.
\item[($\partial$2)] There are the boundaries of the first interval, which corresponds to setting the parameter $s$ to either $0$ or $1$. These two boundary strata do not contribute, since the restriction of $\Gamma$ to them has rank at most $1$, but we integrate a wedge-product containing the pull back of a $2$-form along $\Gamma$ as one of its factors.
\item[($\partial$3)] There are the codimension one boundary strata of the simplex $\Delta_{(m+n+1)}$.
These yield integrals over $\Delta_{(m+n)}\times I \times I$ of forms obtained by multiplying two consecutive factors in $\tilde{\Omega}_{m,n}$.
\end{itemize}

We claim that the contributions from the codimension one boundary strata of $\Delta_{(m+n+1)}$ ($\partial$3) cancel
exactly with the ones obtained from (D1). One can easily see that, in fact, this is the case up to signs.
For a careful discussion of signs, which can be adapted to the current setting, we refer the interested reader to \S3.1 of
\cite{AS}, and in particular to the proof of Theorem 3.10 therein.

Apart from $Z(\Gamma)-X(\Gamma)$, the only other remaining terms in Stokes Theorem come from (D1) and are given by
the sum of integrals of $\pm H_{(k+m+n+2)}^*((-1)^k U_{k,m,n}+(-1)^{k+m+1}V_{k,m,n})$ over $\Delta_{k+m+n+2}\times I \times I$, where
\begin{eqnarray*}
U_{k,m,n} &=& \underbrace{p_1^*(\Gamma\circ r)^*\alpha^1 \wedge \cdots \wedge p_k^*(\Gamma\circ r)^*\alpha^1}_{k} \wedge 
p_{k+1}^*(\Gamma\circ r)^*[\partial,\alpha^2] \wedge \cdots \\
&& \cdots \wedge \underbrace{p_{k+2}^*(\Gamma\circ r)^*\alpha^1 \wedge \cdots \wedge p_{k+m+1}^*(\Gamma\circ r)^*\alpha^1}_{m} \wedge p_{k+m+2}^*(\Gamma \circ r)^*\alpha^2 \wedge \cdots \\
&& \cdots \wedge \underbrace{p_{k+m+3}^*(\Gamma\circ r)^*\alpha^1 \wedge \cdots \wedge p_{k+m+n+2}^*(\Gamma\circ r)^*\alpha^1}_{n},\\
V_{k,m,n} &=& \underbrace{p_1^*(\Gamma\circ r)^*\alpha^1 \wedge \cdots \wedge p_k^*(\Gamma\circ r)^*\alpha^1}_{k} \wedge 
p_{k+1}^*(\Gamma\circ r)^*\alpha^2 \wedge \cdots \\
&& \cdots \wedge \underbrace{p_{k+2}^*(\Gamma\circ r)^*\alpha^1 \wedge \cdots \wedge p_{k+m+1}^*(\Gamma\circ r)^*\alpha^1}_{m} \wedge p_{k+m+2}^*(\Gamma \circ r)^*[\partial,\alpha^2] \wedge \cdots \\
&& \cdots \wedge \underbrace{p_{k+m+3}^*(\Gamma\circ r)^*\alpha^1 \wedge \cdots \wedge p_{k+m+n+2}^*(\Gamma\circ r)^*\alpha^1}_{n}.
 \end{eqnarray*}
We now see that the difference $Z(\Gamma)-X(\Gamma)$ can be written as the sum of integrals over $\Delta_{k+m+n+2}\times I \times I$ of the
forms $[\partial,W_{k,m,n}]$ given by
\begin{eqnarray*}
W_{k,m,n} &=& \underbrace{p_1^*(\Gamma\circ r)^*\alpha^1 \wedge \cdots \wedge p_k^*(\Gamma\circ r)^*\alpha^1}_{k} \wedge 
p_{k+1}^*(\Gamma\circ r)^*\alpha^2 \wedge \cdots \\
&& \cdots \wedge \underbrace{p_{k+2}^*(\Gamma\circ r)^*\alpha^1 \wedge \cdots \wedge p_{k+m+1}^*(\Gamma\circ r)^*\alpha^1}_{m} \wedge p_{k+m+2}^*(\Gamma \circ r)^*\alpha^2 \wedge \cdots \\
&& \cdots \wedge \underbrace{p_{k+m+3}^*(\Gamma\circ r)^*\alpha^1 \wedge \cdots \wedge p_{k+m+n+2}^*(\Gamma\circ r)^*\alpha^1}_{n}.
\end{eqnarray*}

Since, up to sign, $[\partial,-]$ and integration commute, we conclude that $Z(\Gamma)$ differs from $X(\Gamma)$ by
a term in $[\partial,\End^{-2}(V)]$.
\end{proof}


\thebibliography{10}

\bibitem{AS}
C. Arias Abad, F. Sch\"atz,
{\em The $\mathsf{A}_\infty$ de Rham theorem and the integration of representations up to homotopy}, Int. Math. Res. Not. 2013 (16): 3790-3850.
\bibitem{AC}
C. Arias Abad, M. Crainic,
{\em Representations up to homotopy of Lie algebroids}, Journal f\"ur die reine und angewandte Mathematik, 663 (2012) 91-126. 

\bibitem{BaezC}
J. Baez, A. Crans, {\em Higher-dimensional algebra VI Lie 2-algebras},  Theory Appl. Categ. {\bf 12} (2004), 492--538.

\bibitem{BaezS}
J. Baez, U. Schreiber, {\em Higher Gauge Theory: 2-connections on 2-bundles}, arxiv:hepth/0412325.


\bibitem{BS}
J. Block and A. Smith, {\em The higher Riemann-Hilbert correspondence}, Adv. in Math, Volume {\bf 252} (2014), 382--405.

\bibitem{Chen}
K.T. Chen,
{\em Iterated path integrals}, Bull. Amer. Math. Soc. {\bf 83} (1977), 831--879.

\bibitem{FariaMikovic}
J. Faria Martins, V. Mikovic,
{\em Lie crossed modues and gauge invariant actions for 2-BF theories}, arxiv: 1006.0903v3 [hep-th].

\bibitem{FariaPicken2}
J. Faria Martins, R. Picken,
{\em On two-dimensional holonomy}, Trans. Amer. Math. Soc., {\bf 362}, (2010), 5657--5695.

\bibitem{FariaPicken}
J. Faria Martins, R. Picken,
{\em The fundamental Gray 3-groupoid of a smooth manifold and local 3-dimensional holonomy based on a 2-crossed module}, Differential Geom. Appl., {\bf 29} (2) (2011), 179--206.


\bibitem{Picken}
J. Faria Martins, R. Picken,
{\em Surface Holonomy for Non-Abelian 2-Bundles via Double Groupoids}, Adv. in Math, Volume {\bf 226}, Issue 4: 3309--3366 (2011).

\bibitem{FariaCirio}
J. Faria Martins, Lucio Simone Cirio,
{\em Categorifying the $sl(2,\mathbb{C})$ Knizhnik-Zamolodchikov Connection via an Infinitesimal $2$-Yang-Baxter Operator in the String Lie-$2$-Algebra}, arxiv.org/abs/1207.1132.
\bibitem{G}
V. K. A. M. Gugenheim,
{\em On Chen's iterated integrals},
Illinois J. Math. Volume {\bf 21}, Issue 3 (1977), 703--715.

\bibitem{I}
K. Igusa,
{\em Iterated integrals of superconnections}, arXiv:0912.0249.




\bibitem{KP}
K. H. Kamps, T. Porter,
{\em $2$-groupoid enrichments in homotopy theory and algebra}, K-theory 25 (2002), 373-409.


\bibitem{SchreiberWaldorf}
U. Schreiber, K. Waldorf, {\em Smooth Functors vs. Differential Forms}, Homology Homotopy Appl. {\bf 13}
(2011), no. 1, 143--203.
\end{document}